\documentclass[12pt]{article}
\usepackage[utf8]{inputenc}
\usepackage[longnamesfirst]{natbib}
\usepackage{mathtools, amsfonts, amssymb,amsthm}
\mathtoolsset{showonlyrefs}
\usepackage{enumerate}
\usepackage{comment}
\usepackage{ifthen}
\usepackage{hyperref}

\usepackage{bbm}

\usepackage{enumerate}
\usepackage{comment}
\usepackage{subcaption}
\usepackage{siunitx}
\usepackage[ruled,vlined]{algorithm2e}

\newtheorem{theorem}{Theorem}
\newcounter{prop}
\newtheorem{proposition}[prop]{Proposition}
\newcounter{cor}
\newtheorem{corollary}[cor]{Corollary}
\newcounter{lem}
\newtheorem{lemma}[lem]{Lemma}
\newcounter{def}
\newtheorem{definition}[def]{Definition}
\newcounter{rem}
\newtheorem{remark}[rem]{Remark}

\newcommand{\EE}{\mathbb{E}}

\newcommand{\PP}{\mathbb{P}}

\usepackage{color}

\newcounter{assumptionH}
\newenvironment{assumptionH}%
{%
\medskip
\noindent\textbf{Assumptions.}
  \begin{enumerate}[({H}1)]%
  \setcounter{enumi}{\value{assumptionH}}%
}{%
  \setcounter{assumptionH}{\value{enumi}}%
  \end{enumerate}
}

\newcounter{assumptionD}
\newenvironment{assumptionD}%
{%
	\medskip
	\noindent\textbf{Assumptions.}
	\begin{enumerate}[({D}1)]%
		\setcounter{enumi}{\value{assumptionD}}%
	}{%
		\setcounter{assumptionD}{\value{enumi}}%
	\end{enumerate}
}

\newcommand{\NN}{\mathbb{N}}
\newcommand{\RR}{\mathbb{R}}

\newcommand{\Yim}{{Y_{i,m}}}

    \newcommand{\Tmi}{\mathbf T_{i,m}}

\newcommand{\Xp}[1]{X_{#1}}
\newcommand{\tobs}{\mathcal T^{\rm obs}}
\newcommand{\Tobs}[1]{\mathcal T_{#1}}
\newcommand{\alphamu}{\alpha_\mu}
\newcommand{\hatalphamu}{\widehat {\alpha}}

\author{Omar Kassi\footnote{Corresponding author. Ensai, CREST - UMR CNRS 9194, France; omar.kassi@ensai.fr} \qquad Valentin Patilea\footnote{Ensai, CREST - UMR CNRS 9194, France; valentin.patilea@ensai.fr}}
\date{\today}
\title{Optimal inference for the mean of random functions}

\usepackage{xr}
\begin{document}
	\maketitle

	\begin{abstract}

We study estimation and inference for the mean of real-valued random functions defined on a hypercube. The independent random functions are observed on a discrete, random subset of design points, possibly with  heteroscedastic noise. We propose a novel optimal-rate estimator based on Fourier series expansions and establish a sharp non-asymptotic error bound in $L^2-$norm. Additionally, we derive a non-asymptotic Gaussian approximation bound for our estimated Fourier coefficients. Pointwise and uniform confidence sets are constructed. 
Our approach is made adaptive by a plug-in estimator for the Hölder regularity of the mean function, for which we derive non-asymptotic concentration bounds.

\medskip 


\medskip 


	\end{abstract}




\section{Introduction}\label{sec:intro}

Modern studies from a variety of fields such as biophysics, neuroscience, environmental science, and medicine, record multiple observations obtained as discretized measurements of curves or surfaces. One example is the Argo project, which has recently become an important source of data for oceanography and climate research \citep[see, e.g., ][]{YSH2022}. The usual modeling paradigm is to consider that the curves or surfaces are realizations of some \emph{random function} defined on a one- or multidimensional domain.  See also \cite{KS2023} and \cite{GRLG2024} for recent surveys. 

The problem of estimating the mean of a  {random function} (sometimes called \emph{random field})  is considered, when the random function is defined on a multidimensional compact domain and several independent realizations are observed at random design points, possibly with heteroscedastic error. The number of design points of each realization of the random function can be as low as 1, or   arbitrarily large. Another aim of this study is  to construct pointwise confidence intervals and confidence bands for the mean function. For these purposes, we propose a new Fourier series-based approach. We construct an optimal rate mean estimator and confidence sets with optimal rate diameter over classes of Hölder functions. In the interest of adaptivity, we also propose an estimator for the Hölder regularity of the mean function.

The estimation of the mean of random functions has been mainly studied in the case of $1-$dimensional domains using smoothing such as  kernels \citep{ZW2016} or splines \citep{CY2011}, \citep{X2020}. Pointwise, $L^2$ and uniform rates of convergence in probability have been derived, with random design or a common grid of design points for all the realizations. \citep{ZW2016} derive also a pointwise asymptotic normality result. Our mean estimation problem is also related to the nonparametric regression with dependent errors, see \cite{Ef1999}, \cite{BCV2001}. 

A variety of methods for simultaneous inference of the mean function with functional data have been proposed recently, with simulation based methods being currently the most successful. See, for example, \cite{D2011}, \cite{CR2018}, \cite{DKA2020}, \cite{WWWO2020}, \cite{TS2022} and \cite{LR2023}. The latter reference includes a clear and informative up-to-date review. Many existing approaches for constructing confidence bands assume the  realizations of the random function fully observed, or at least, they are observed on a sufficiently dense grid of fixed design points such that the discretization error is negligible. In such cases, the mean estimator  converges to a Gaussian process and the inference is straightforward. Related to the inference problem we study is the extensively studied topic of adaptive confidence bands in nonparametric regression, see for example \cite{HN2011}, \cite{CBadapt2014}. It appears that the study of adaptive confidence regions and the estimation of regularity for the mean of random functions have not yet been addressed.

The paper is organized as follows. In Section \ref{sec:problem}, the definitions related to multivariate random functions are introduced and the functional data with random design is described. Moreover, we introduce the \emph{de La Vall\'ee Poussin series} and the relevant results emanating from the approximation theory of real-valued functions defined on a multidimensional unit cube. The new idea of estimating the Fourier coefficients through a fast integration technique based on \emph{control variates} is introduced in Section \ref{new_fo_coe}. Non-asymptotic bounds on the quadratic risk on such Fourier coefficients are derived. The mean function estimator based on de La Vall\'ee Poussin series, with fast integration-based estimates of the Fourier coefficients, is studied in Section \ref{sec:mean_fct_rate}. A non-asymptotic bound on the $L^2-$error of our estimator is derived, with explicit constants, and used to define the \emph{optimal} mean estimator. 
Our bound seems to be new,  requires no assumption about the number of design points, and holds \emph{uniformly} over increasing sets of Hölder  functions which are limits of Fourier series and have the uniform norm less than a diverging sequence of numbers determined by the sample size.
The $L^2-$error bound determines the phase transition between the \emph{dense} and \emph{sparse regime}, which in our context depends on the regularity of the mean function. Within the dense regime, the number of random design points is such that the optimal mean estimator attains the parametric rate of the ideal empirical estimator based on fully observed random functions. On the contrary, within the  sparse regime, the number of random design points is too low, potentially as low as one point per random function, thereby rendering  the parametric rate unattainable. Section \ref{sec:infer} is dedicated to the inference for the vectors of Fourier coefficients. A non-asymptotic bound on the \emph{Gaussian approximation} of our Fourier coefficients estimator is derived. Consequently, \emph{pointwise confidence intervals} and \emph{confidence bands} for the mean function are obtained. Our Gaussian approximation and the inference results appear to be new in the context of random functions. The critical values for the confidence regions can be derived from the asymptotic variance of the Gaussian approximation. Our estimation and inference approach for the mean function crucially depends on the random nature of the design. In this paper, we assume that we know the density of the design, but the results can be easily adapted to allow for an estimated design density, as long as it is smoother than the mean function. Section \ref{sec:reg} is dedicated to the estimation of the Hölder regularity of the mean function, the objective being to make the mean estimation and inference \emph{adaptive}. First, we propose to use the concept of \emph{local Hölder regularity} in order to make the mean Hölder exponent identifiable. Assuming that the mean function belongs to a set of functions with constant local Hölder regularity exponent, we propose a simple estimator of this regularity parameter when its value is less than 1. We derive non-asymptotic bounds for the exponent estimator, which can be next plugged into the expression of the optimal mean function estimator to obtain optimal \emph{data-driven} estimation and inference for the mean function. The \emph{regularity estimation} results are new, of independent interest in the context of adaptive nonparametric methods for random functions, and crucially exploit the nature of the functional data. The proofs of the most important results are given in the Appendix, while complements and more technical details are gathered in a Supplement. In that Supplement, we also propose a \emph{subsampling-}based method for constructing confidence regions that avoids computing the asymptotic variance of the Gaussian approximation.

\section{The setup}\label{sec:problem}

Let $\mathcal T = [0,1]^D$, with $D\geq 1$, be the domain of the real-valued random function $X$. The random copies of $X$, denoted  $X_i$, $i=1,2,..,N,$ are referred to as trajectories or sample paths, and are supposed to be continuous.  It is also assumed that the mean and the covariance functions of $X$, denoted by $\mu(\cdot)$ and $\gamma(\cdot, \cdot)$ respectively, exist and are continuous. They are defined as  
\begin{equation}
	\mu(\mathbf t)= \EE [X(\mathbf t)] \quad \text{and} \quad \gamma(\mathbf s,\mathbf t)= \mathbb E\left[ \{X(\mathbf s) - \mu(\mathbf s)\}  \{X(\mathbf t)- \mu(\mathbf t)\} \right],\quad \text{ for any } \mathbf s,\mathbf t\in \mathcal T.
\end{equation}
The following notations will be used~:  $\mathcal T^2 = \mathcal T \times \mathcal T $; for any $a\in\mathbb R$, $\lfloor a\rfloor$ denotes the largest integer smaller than  $a$;  $\lesssim$ means that the left side is bounded by a positive constant times the right side;
for $a,b\in\RR$, $a \asymp b $ means  $a \lesssim b $ and  $b \lesssim a $; the vectors are column matrices and, for any vector $\mathbf a$, $\mathbf a^\top$ and $\|\mathbf a\|$ denote its transpose and the Euclidean norm, respectively; 
for any $\mathbf k=(k_1,\ldots,k_D)\in \mathbb N^D$,  $|\mathbf k |_1:=k_1+\dots + k_D$; for any continuous real-valued function $g$, $\|g\|_\infty$ denotes its uniform norm.  

In this section, we first describe the type of data used for estimation and inference. We next introduce Fourier partial sums and the associated de La Vallée Poussin sums, which improve on the approximation error of standard Fourier sums.

\subsection{Data}\label{sec:data}

For each $1\leq i \leq N$, let $M_i$ be a positive integer and $\mathbf  T_{i,m} \in \mathcal T$, $1\leq m\leq M_i$, be the observed design (domain) points for  $X_i$. The data associated with the trajectory $X_i$  consists of  the pairs  $(\Yim , \mathbf  T_{i,m}  ) \in\mathbb R \times \mathcal T $ where 
\begin{equation}\label{model_eq}
	\Yim = \Xp{i} (\mathbf T_{i,m} ) + \varepsilon_{i,m}, \quad \text{with}  \quad  \varepsilon_{i,m}= \sigma(\mathbf T_{i,m} ) {e}_{i,m}, 
	\quad  1\leq m \leq M_i,
\end{equation}
with $e_{i,m}$ independent copies of a centered variable $e$ with unit variance, and $\sigma(\cdot)$ is an unknown function that   accounts for possibly heteroscedastic measurement errors. The  design points $\mathbf T_{i,m}$ are  random copies of the vector $\mathbf  T\in \mathcal T$ which admits a density $f_{\mathbf  T}$ with respect to the Lebesgue measure. This is the so-called \emph{random design} framework. 
Let
\begin{equation}\label{def_tobs}
	\tobs =  \bigcup_{1\leq i \leq N}\Tobs{i} \qquad \text{ with } \qquad  \Tobs{i} = \{\mathbf T_{i,m}:   1\leq m \leq M_i\}.
\end{equation}  
Finally, the number of all data points $(\Yim , \mathbf T_{i,m} )$ used to estimate the mean $\mu(\cdot)$ is $\overline M=M_1+\cdots + M_N$. For simplicity,  $M_1,\ldots, M_N$ are considered non-random. 

The model equation \eqref{model_eq} can be rewritten as a nonparametric regression problem 
\begin{equation}\label{model-eq2}
\Yim = \mu( \mathbf  T_{i,m})+\{\Xp{i} - \mu\} (\mathbf  T_{i,m}) + \varepsilon_{i,m} := \mu(\mathbf  T_{i,m}) + \eta_{i,m}, \quad 1\leq i \leq N,  \; 1\leq m \leq M_i.
\end{equation} 
If  $e$, $\mathbf T$ and  $X$ are independent, the error term $\eta_{i,m}$  has the following  conditional covariance structure given the design points~:  
\begin{equation}\label{structure_dep}
\mathbb E\left[\eta_{i_1, m_1} \eta_{i_2, m_2}\mid \tobs \right]=
 \big\{ \gamma(\mathbf T_{i_1,m_1}, \mathbf T_{i_2,m_2})
  + \delta_{ m_1, m_2} \sigma^2 (\mathbf T_{i_1,m_1}) \big\}
\delta_{ i_1, i_2},
\end{equation}
where $\delta_{ m_1, m_2} = 1$  if $ m_1= m_2$ and zero otherwise.

\subsection{\!\!Multivariate Fourier series and de La Vallée Poussin operator}\label{subsec:legendre}

We consider the Fourier basis on $[0,1]$ defined by the elements
\begin{equation}\label{fourier_sys} 
\phi_0(t) = 1, \quad   \phi_{2k-1}(t) = \sqrt{2} \sin(2\pi k t), \quad \phi_{2k}(t) = \sqrt{2} \cos(2\pi k t),   \quad t \in [0,1],\; k = 1, 2, \dots.
\end{equation}
The elements of the basis on $\mathcal T$ are defined as follows~: for  $\mathbf{k} = (k_1, k_2, \dots, k_D)\in \mathbb N^D$, 
$$
\phi_{\mathbf{k}}(\mathbf{t}) = \prod_{j=1}^D \phi_{k_j}(t_j), \quad \text{with }\;\;  \mathbf{t} = (t_1, \dots, t_D)\in \mathcal T=[0,1]^D.
$$
Let $\mathcal F _D$ be the closed linear span of the Fourier basis, that is
	$\mathcal F_D = \overline{\operatorname{Span}}\left\{ \phi_{\mathbf{k}}: \mathbf{k}\in \mathbb N ^D\right\}$.

\medskip

\begin{definition}\label{holder_def}
	Let the set $ \Sigma(\alpha, C;D) $, with $\alpha,C>0$, be the Hölder space of
	functions $g:\mathcal T \rightarrow \RR$ for which all order $\lfloor \alpha\rfloor$ partial derivatives exist and, for any $\boldsymbol{\alpha} \in \mathbb N^D$ with $|\boldsymbol \alpha |_1=\alpha_1+\alpha_2+\dots + \alpha_D = \lfloor \alpha\rfloor,$ it holds that
	\begin{equation}\label{unif_holdr_cdt}
		\left| \partial^{\boldsymbol  \alpha} g(\mathbf t)-\partial ^{\boldsymbol \alpha }g(\mathbf s) \right| \leq C \|\mathbf t- \mathbf s\|^{\alpha- \lfloor \alpha \rfloor}, \qquad \forall \mathbf t, \mathbf s \in \mathcal T ,
	\end{equation}
	where $\partial ^{\boldsymbol{\alpha}}= \partial _1^{\alpha_1}\partial _2^{\alpha_2} \dots \partial _D^{\alpha_D}  $ is the multi-index partial derivatives operator.
\end{definition}

\medskip

The  $L$-th partial sum of the Fourier series for the function $\mu$ is defined   as
\begin{equation}\label{triang_partial}
S_L(\mu, \mathbf{t}) = \sum_{|\mathbf{k}|_1 \leq 2 L} a_{\mathbf{k}} \phi_{\mathbf{k}}(\mathbf{t}), \qquad \text{with Fourier coefficients } \;\;\; a_{\mathbf{k}} = \int_{\mathcal{T}} \mu(\mathbf{t}) \phi_{\mathbf{k}}(\mathbf{t}) \, \mathrm{d}\mathbf{t}.
\end{equation}
 $ S_L(\mu, \cdot) $ is also called a \emph{triangular}   partial sum. Alternatively,  the \emph{rectangular}  $L$-th partial sum can be considered, obtained by replacing $ |\mathbf k|_1$ with $ |\mathbf k|_\infty=\max_{1\leq j \leq D}k_j$. Our results   can be easily adapted to use rectangular partial sums instead of the triangular ones. See \cite{nemeth2014} for a discussion of the two types of partial sums.

Next, we consider the \emph{de La Vallée Poussin} sum, defined as~:
\begin{equation}\label{def:poussin}
V_L(\mu, \mathbf{t}) = \frac{1}{L} \left\{ S_L(\mu, \mathbf{t}) + S_{L+1}(\mu, \mathbf{t}) + \dots + S_{2L-1}(\mu, \mathbf{t}) \right\},\quad \mathbf{t} \in \mathcal{T}.
\end{equation}
For a fixed integer $L$, the operators 
$
S_L: g \mapsto S_L(g, \cdot)$   and  
$V_L: g \mapsto V_L(g, \cdot)$
are well-defined on the space of continuous, real-valued functions defined on $\mathcal{T}$, defining a projection operator onto the linear subspace $\mathcal{V}_L = \operatorname{Span}(\phi_{\mathbf{k}} : |\mathbf{k}|_1 \leq 4L-2 )$. We denote by $\|\cdot\|_{\rm op}$ the operator norm associated with the uniform norm, that is
$$
\|S_L\|_{\rm op} = \sup_{\|g\|_\infty \leq 1} \|S_L(g, \cdot)\|_\infty, \qquad \|V_L\|_{\rm op} = \sup_{\|g\|_\infty \leq 1} \|V_L(g, \cdot)\|_\infty.
$$
For a projection $P$, the uniform approximation error $\|g - P(g)\|_\infty$ can be bounded using the operator norm as follows~:
\begin{equation}\label{dist_best}
\|g - P(g)\|_\infty \leq \| \text{id} - P \|_{\rm op} \operatorname{dist}(g, \mathcal{V}_L) \leq (1 + \|P\|_{\rm op}) \operatorname{dist}(g, \mathcal{V}_L),
\end{equation}
where $ \operatorname{dist}(g, \mathcal{V}_L) = \inf_{h\in \mathcal V_L} \|h-g \|_\infty$ is the distance between $g$ and the subspace $\mathcal V_L$.
It was shown that the exact orders of the operator norms of $S_L$ and $V_L$ are
$$
 \|S_L\|_{\rm op} \asymp (\log L)^D \quad \text{ and } \quad \|V_L\|_{\rm op} \asymp \left(\log2\right)^D.
$$ 
See \citet[Th. 2.1]{szili2009}, \citet[Th. 2.2]{nemeth2014}. These rates indicate that the \emph{de La Vallée Poussin} operator $V_L$ leads to more accurate approximations than $S_L$. 

\medskip

\begin{proposition}\label{reste-control}
	Assume $\mu \in \Sigma (\alphamu, C_\mu; D) \cap \mathcal F_D$ with $\alphamu$  and $C_\mu$ are two positive numbers. Then, a positive constant $\mathcal C_\mu$ exists, depending only on $\alpha_\mu$, such that 
	$$
	|V_L(\mathbf t)- \mu(\mathbf t) | \leq \mathcal C_\mu L^{- \alphamu},\qquad \forall \mathbf  t\in\mathcal T.
	$$
\end{proposition}

\medskip

In the case $\alpha_\mu \in(0,1]$, Proposition \ref{reste-control} is a direct  consequence of \eqref{dist_best} combined with 
\begin{equation}\label{handcomb}
\operatorname{dist} (\mu, \mathcal V_L) \leq C \sum_{j=1}^{D} \omega_j \left ( \mu, (2L+1)^{-1}\right),
\end{equation}
where $C$ is some constant independent of $\mu$ and $L$, and   
$$ 
\omega_j(\mu, z) := \sup_{ | t_j - s_j|\leq z} | \mu( t_1, \dots, t_j, \dots t_D) -\mu( t_1, \dots, s_j, \dots t_D)  |,\quad 1\leq j\leq D, 
$$
is the partial modulus of continuity. The inequality \eqref{handcomb} extends to the case where the mean function admits partial derivatives. For the justification of \eqref{handcomb} and the extension to smoother functions, see \citet[Section 5.3.2]{timan_book},  \citet[equations (2) and (3) page 192]{scomb66},   \citet[equation (22)]{mason1980near}. For the univariate case ($D=1$), see \citet[Ch. 7]{devore1993constructive} and   \citet[Proposition 2.4.2]{efro2008}.


\section{A new approach to estimating Fourier coefficients}\label{new_fo_coe}

The Fourier coefficients $a_{\mathbf k}$ are integrals that can be approximated using ideas from the Monte Carlo integration literature.  We consider the control neighbors approach, a linear integration rule studied by \cite{leluc2024speeding}. See also \cite{oates2017control}.

\subsection{Linear integration}

To describe the linear integration rule, let $\varphi:\mathcal T \rightarrow \mathbb R$ be a $\beta-$Hölder continuous function, $\beta \in(0,1]$, for which we want to approximate
\begin{equation}\label{def_Iphi}
	I(\varphi) = \int_{\mathcal T}\varphi(\mathbf t)  f_{\mathbf T}(\mathbf t) {\rm d}\mathbf t,
\end{equation}
using the values $\varphi(\mathbf T_l)$, $1\leq l \leq n$, with $\mathbf T_l$ random copies of a continuous variable $\mathbf T$ with density $f_{\mathbf T}$ with the support $\mathcal T$. Let $\widetilde \varphi$ be a generic approximation of $\varphi$, the so-called \emph{control variate}, such that $I(\widetilde \varphi)$ can be calculated, where $I(\widetilde \varphi)$ is the integral in \eqref{def_Iphi} with $\widetilde \varphi$ instead of $\varphi$. The estimate of $I(\varphi)$ is then defined as
\begin{equation}\label{def_Iphi_hat}
	\widehat {I }(\varphi)= I(\widetilde \varphi) + \frac{1}{n} \sum_{l=1}^n  \{\varphi(\mathbf T_l) - \widetilde \varphi(\mathbf T_l) \}.
\end{equation}
By construction, 
	$\mathbb E \big[ \widehat {I }(\varphi) \big] = I(\varphi)$  and   $\operatorname{Var} \big[ \widehat {I } (\varphi)\big] \lesssim \|\varphi - \widetilde \varphi\|^2_\infty\;  n^{-1}$.
With  suitable  $\widetilde \varphi$, the regularity of the integrand function $\varphi$ then contributes to the accuracy of   $\widehat I(\varphi)$, thus leading to  
\begin{equation}\label{fast_rateNN}
	\mathbb E \left( \big|\widehat {I } (\varphi)- I(\varphi)\big|^2\right)^{1/2} \lesssim  n^{-1/2 }n^{ - \beta/D}, 
\end{equation}
which is known to be optimal. See \cite{novak2016some}, \cite{BAKHVALOV2015502}. 

\cite{leluc2024speeding} propose a piecewise constant control variate function $\widetilde \varphi$ based on the leave-one-out $1-$nearest neighbor ($1-$NN$-$loo) estimates. In this case, the control variate function is $\widetilde \varphi^{(l)}$, as it depends on $l$, and \eqref{def_Iphi_hat} becomes
\begin{equation}\label{def_Iphi_hat2}
	\widehat {I }(\varphi)=  \frac{1}{n} \sum_{l=1}^n \left[\varphi(\mathbf T_l) +\{   \widetilde \varphi^{(l)}(\mathbf T_l) -I(\widetilde \varphi^{(l)})  \}\right].
\end{equation} 
An appealing feature of the integral approximation by control neighbors is that $\widehat {I}(\varphi) $ is a convex combination of the integrant values $\varphi(\mathbf T_l)$  with explicit weights depending only on the design points $\mathbf T_l$, $1\leq l \leq n$. More precisely,  the weights $w_{l}$ exist such that 
\begin{equation}\label{lin_rep1}
	\widehat {I } (\varphi)= \sum_{l=1}^n w_{l} \varphi(\mathbf T_l), \qquad \text{with} \quad \sum_{l=1}^n w_{l} =1.	
\end{equation}
The precise description of the  control neighbor and the weights $w_{l}$ in the linear representation \eqref{lin_rep1} is given in the following using few more definitions  from \cite{leluc2024speeding}.

\medskip

\begin{definition}[Leave-one-out neighbors, and Voronoi cells and volumes]\label{def_generic_NN}
Let $\mathfrak S=\{\mathbf T _l: 1\leq l \leq n \} \subset \mathcal T$
be a  set of points in the  domain $\mathcal T\subset \mathbb R^D$, $D\geq 1$. 
\begin{enumerate}

\item  The nearest neighbor of $\mathbf t\in \mathcal T$ among  $\mathfrak S\setminus \{\mathbf T_l\}$, called the \emph{leave-one-out nearest neighbor}, is 
$$
\widehat N(\mathbf t;\mathfrak S\setminus \{\mathbf T_l\})= \arg\min_{\mathbf s\in \mathfrak S \setminus \{S_l\}}  \|\mathbf t- \mathbf  s\|.
$$
When the $\arg\min$ is not unique, $\widehat N(\mathbf t;\cdot)$ is chosen by lexicographic order.

\item For $ \mathbf T_l \in \mathfrak S$, let
$\mathcal S_{j} (\mathfrak S\setminus \{\mathbf T_l\}) =  \{\mathbf t\in \mathcal T: \widehat N(\mathbf t;\mathfrak S\setminus \{\mathbf T_l\})= \mathbf T_j \} $ be the \emph{leave-one-out Voronoi cell}.

\item Assume that $\mathbf T_l$ are i.i.d. random variables admitting a density $f$ on $\mathcal T$. Then the volumes associated with the leave-one-out Voronoi cells are
$$
V_{j}(\mathfrak S \setminus \{\mathbf T_l\})  = \int_{\mathcal S_{j}(\mathfrak S\setminus \{\mathbf T_l\})} f(u) {\rm d}u
, \qquad 1\leq j \leq n.
$$
\end{enumerate}
\end{definition}


The $1-$NN$-$loo control variate function (also called a control neighbor) is defined by $ \widetilde \varphi ^{(l)}(\mathbf t)= \varphi(\widehat N(\mathbf t;\mathfrak S\setminus \{\mathbf T_l\}))$. Let $\mathbbm 1 \{\cdot\}$ denote the indicator function. 

\medskip

\begin{definition}[Degree and cumulative length]\label{def_generic_d}
Let $\mathfrak S $ be like in Definition \ref{def_generic_NN}.
The degree $\widehat d_j$  represents the number of times $\mathbf T_j$ is a nearest neighbor of a point $\mathbf T_l$ for $l\ne j$. The associated $j$-th cumulative Voronoi volume is denoted by $\widehat c_j$, that is  
\begin{equation}\label{def_c_et_d}
\widehat d_j =\widehat d_j(\mathfrak S)= \sum_{l:l\ne j}\mathbbm 1 \{\mathbf T_l \in \mathcal S_j(\mathfrak S\setminus \{ \mathbf T_l\})\} \quad \text{and }\quad \widehat c_j= \widehat c_j(\mathfrak S)= \sum_{l:l\ne j}V_j(\mathfrak S\setminus\{ \mathbf T_l\}).
\end{equation}
\end{definition}

\medskip

The weights in the representation \eqref{lin_rep1} are shown to be $w_l= (1+\widehat c_l - \widehat d_l)/n$ \citet[][Proposition 1]{leluc2024speeding}. Explicit formulae of $\widehat c_l $ and $\widehat d_l $ when $\mathcal T = [0,1]$ are provided in the Supplement. We now present some properties of $\widehat c_l $ and $\widehat d_l $.

\medskip

\begin{proposition}\label{lem_Vor_Index}
Let $\mathfrak S$ in Definition \ref{def_generic_NN} be the independent sample $\mathbf T_1,\ldots,\mathbf T_n$ of a random vector  $\mathbf T\in\mathcal T  $, and let $A (\cdot)$ be an integrable function of $\mathbf T$.   Then,  
$$
\mathbb{E} \left[ \widehat{d}_{l} \mid \mathbf T_l \right] = \mathbb{E}[ \widehat c_{l} \mid \mathbf T_l ],\qquad \mathbb{E} \left[A (\mathbf T_l)\widehat{d}_{l}\right] = \mathbb{E}[A (\mathbf T_l)\widehat c_{l}] \quad \text{ and } \quad \mathbb{E}  [ \widehat{d}_{l} ] = \mathbb{E}[ \widehat c_{l}]=1.
$$ 
\end{proposition}


\subsection{The coefficients estimators}
Let $\mathbf k \in \NN^D$. To estimate the coefficients $a_{ \mathbf k} = \int_{\mathcal T}\mu( \mathbf t)\phi_{\mathbf k}(\mathbf t){\rm d}\mathbf t$, we aggregate all available data points. 
Then,   $\mathfrak S=\tobs$  in Definition \ref{def_generic_NN}  and a pair index $(i,m)$ corresponds to an index $j$ or $l$  in Definition \ref{def_generic_d}. The cardinal of the set of pairs  $(i,m)$ is $\overline M= \sum_{i=1}^N M_i$. Define
\begin{equation}\label{estimator:poolling}
\widehat a_{\mathbf k}= \sum_{i,m} \omega_{i,m}\frac{\Yim \phi_{\mathbf k}(\mathbf T_{i,m})}{f_{\mathbf T}(\mathbf T_{i,m})},\qquad \text{ where } \quad \omega_{i,m}= (1 +\widehat c_{i,m} - \widehat d_{i,m}  )/\overline M, 
\end{equation}
$ 1\leq m\leq M_i$, $1\leq i\leq N$, with  $\widehat c_{i,m} =\widehat c_{i,m}(\tobs)$ and $\widehat d_{i,m} =\widehat d_{i,m}(\tobs)$ defined as in \eqref{def_c_et_d}. 

\medskip

\begin{assumptionH}

	\item \label{assuH2} The $ \Tmi$'s are random copies of $\mathbf T\in \mathcal T$, and $\overline M :=\sum_{i=1}^N M_i \geq 4$.
	
	\item \label{assuH12}  $\mathbf T\in \mathbb R^D$ admits a density $f_{\mathbf T}$; constants $C_0, C_1 $ exist such that 
	$
	0< C_0 \leq  f_{\mathbf T} \leq C_1 .
	$

	\item \label{assuH3} The mean function $\mu$ belongs to the Hölder class $\Sigma (\alphamu, C_{\mu};D)$, for some $\alphamu, C_{\mu}>0$. 
\end{assumptionH}

\smallskip

To derive bounds for the quadratic risk of the $\widehat a_{\mathbf k}$, we decompose it as follows~:
\begin{equation}\label{estimator:decom}
	\widehat a_{\mathbf k}= \sum_{(i,m)} \omega_{i,m}\frac{\mu(\mathbf T_{i,m}) \phi_{\mathbf k}(\mathbf T_{i,m})}{f_{\mathbf T}(\mathbf T_{i,m})}+ \sum_{(i,m)} \omega_{i,m}\frac{\eta_{i,m} \phi_{ \mathbf k}(\mathbf T_{i,m})}{f_{\mathbf T}(\mathbf T_{i,m})}:= B_{ \mathbf k}+V_{ \mathbf k},
\end{equation}
with $\eta_{i,m}$ is defined in \eqref{model-eq2}. 
In the following result, a version of Theorem 1 in \cite{leluc2024speeding} proved in the Supplement, we show that $B_{\mathbf k} $ is  unbiased and we bound its risk.

\medskip

\begin{proposition}\label{variance-B}
	Under assumptions (H\ref{assuH2}), (H\ref{assuH12}) and (H\ref{assuH3}), we have $\mathbb E (B_{\mathbf k}) = a_{ \mathbf k}$ and
	\begin{equation}
		\EE\left[ (B_{\mathbf k}- a_{\mathbf k})^2\right] \leq K_{\rm NN-loo}^22^{D+1}\left\{C_\mu^2 \overline M^{\; -1 -2\alphamu/D} + (2\pi)^2  \|\mu \|_\infty^2 |\mathbf k|_1^2 \overline M^{\; -1 - 2/D} \right\},
	\end{equation}
	where $K_{\rm NN-loo}$ depends only on $C_0,C_1, D$ and $\alpha_\mu$. 
\end{proposition}

\smallskip

In this paper we assume that $f_{\mathbf T}$ is given and, similarly to \citep[pp. 8]{RV2006}, `this is not merely a technical assumption, but essential' for our approach. In many practical situations it is realistic to assume that $f_{\mathbf T}$ is smoother than $\mu$, in which case our conclusions readily accommodate an estimated $f_{\mathbf T}$ instead of $f_{\mathbf T}$.  

Before bounding the quadratic risk of the stochastic part $V_{ \mathbf k}$ in the decomposition of $\widehat a_{ \mathbf k}$, let us introduce the following additional mild assumptions.

\medskip

\newpage
\begin{assumptionD} \nobreak	\item \label{assuD1}  The curves $X_i$ are independent sample paths of a  second order  process $X$ with mean function $\mu$ and covariance function $\gamma$. Moreover, $\int_{\mathcal T} \gamma (\mathbf t,\mathbf t){\rm d}\mathbf t<\infty$.
	
	\item \label{assuD2}  The heteroscedastic errors are $\varepsilon_{i,m}=\sigma(\mathbf T_{i,m})  e_{i,m}$ with $e_{i,m}$ independent copies of a centered variable $e$ with unit variance. 
	
	\item \label{assuD3}  The sample paths of $X$, the realizations of $\mathbf T$ and $e$ are mutually independent. 
	
	\item \label{assuD4} The density  $f_{\mathbf T}(\cdot)$ and the conditional  variance function $\sigma^2(\cdot)$ of the errors  are $\alpha_f-$Hölder  and $\alpha_\sigma-$Hölder continuous functions, respectively, for some $\alpha_f, \alpha_\sigma >0$.
	
	\item \label{assuD5} There exist $H_\gamma\in(0,1]$ and $L_\gamma>0$ such that, for any $\mathbf s\in\mathcal T$, the  function $\mathbf t\mapsto \gamma(\mathbf s,\mathbf t)$ is a $H_\gamma-$Hölder   continuous function with Hölder constant $L_\gamma$.

\end{assumptionD}

\medskip

Under the conditions (D\ref{assuD1}) to (D\ref{assuD3}), we have $\EE[ V_{ \mathbf k}]=0$ and
$$
\EE[ (\widehat a_{ \mathbf k}-a_{ \mathbf k})^2] = \EE[ (B_{ \mathbf k}- a_{ \mathbf k})^2] + \EE[ V_{ \mathbf k}^2].
$$
We can now derive the  risk bound of the stochastic part in the Fourier coefficient estimator.  

\medskip

\begin{proposition}\label{variance-V}
Assume that  (H\ref{assuH2}), (H\ref{assuH12}), and (D\ref{assuD1}) to (D\ref{assuD5}) hold true. 
Then,  $\forall \mathbf k\in \NN^D$,
\begin{multline}
\EE[V_{\mathbf k}^2] = \frac{\varrho(D)}{\overline M} \int_{\mathcal T} \frac{ \{\sigma^2(\mathbf t)+\gamma(\mathbf t,\mathbf t)\}\phi_{\mathbf k}^2(\mathbf t) }{f_{\mathbf T}(\mathbf t)}{\rm d}\mathbf t\times\left \{ 1  + o(1) \right \}
\\
+\frac{ \sum_{i=1}^N M_i(M_i-1)}{\overline M(\overline M - 1)}\int_{\mathcal T^2} \gamma(\mathbf s, \mathbf t)\phi_{\mathbf k}(\mathbf s)\phi_{\mathbf k}(\mathbf t) {\rm d}\mathbf s{\rm d}\mathbf t \times\{ 1 + o(1)\}),
\end{multline}
where  $\varrho(D)$ is a constant depending only on the dimension $D$.
The constants in the $o(\cdot)$  quantities do not depend on $\mathbf k$. 
\end{proposition}

\medskip

For $D=1$, $\varrho(D)=5/2$ (see the Supplementary Material). For $D>1$, $\varrho(D)$ is determined by the maximum number of nearest neighbors of a given point in the sample of $\mathbf T_{i,m}$'s and the second order moments of the Voronoi volumes. Details are given in the proof of Proposition \ref{variance-V} in the Appendix.   Finally, note that if $  N M_i/\overline M  \asymp 1$, then
 $$
\frac{ \sum_{i=1}^N M_i(M_i-1)}{\overline M(\overline M - 1)} \asymp N^{-1}.
 $$ 
 
 
\begin{remark}
If the observation points $\{\mathbf T_{i,m}\}$ do not change across the trajectories $X_i$, and they are on a fixed, regular grid of, say, $\mathfrak m$ points, then the weights in \eqref{estimator:poolling} are all equal. For simplicity, let $D=1$.
The estimate of the $\mathbf k-$th Fourier coefficient then becomes 
$$
\overline{a}_{\mathbf k} = \frac{1}{N \mathfrak{m}} \sum_{i=1}^{N} \sum_{m=1}^{\mathfrak{m}} Y_{i,m} \phi_{\mathbf k}\left (\frac{m}{\mathfrak m}\right),
$$
and is no longer unbiased. See also the discussion in \citep[Ch. 4]{efro2008}.

\end{remark}


\section{Mean function estimation and risk bounds}\label{sec:mean_fct_rate}

The estimator of the mean function $\mu$ is given by the sample version of the de La Vallée Poussin approximation $V_L(\mu, \cdot)$  in \eqref{def:poussin},
with the Fourier coefficients $\widehat a_{\mathbf k }$ are defined in \eqref{estimator:poolling}. With $\widehat S_L(\mu, \mathbf t)$ defined in  \eqref{triang_partial}, our mean function estimator is then
\begin{equation}\label{def:mean_est}
\widehat \mu(\mathbf t)=\widehat V_L (\mu;\mathbf t)= \frac{1}{L} \left\{ \widehat S_L(\mu, \mathbf t)+\widehat S_{L+1}(\mu, \mathbf t)+\dots+\widehat S_{2L-1}(\mu, \mathbf t) \right\},
\end{equation}
where $L$ is an integer to be chosen. 
Let  $B_{\mathbf k}$, $V_{ \mathbf k}$ be as  in \eqref{estimator:decom}, $\mathbf k \in \NN^D$. Then,
\begin{equation}\label{ineq_SM_P5}
\EE\left\| \widehat  V_L (\mu;\cdot)-  V_L (\mu;\cdot)  \right\|_2^2 
	\leq  \frac{1}{L} \sum_{j=0}^{L-1} \sum_{|\mathbf k|_1\leq 2(L+j)}\left\{\EE\left[(B_{ \mathbf k} - a_{\mathbf k})^2\right] + \EE \left[V_{\mathbf k}^2\right] \right\}.
\end{equation}
Details for deriving \eqref{ineq_SM_P5} are given in the Supplement.  Here, $\|\cdot\|_2$ denotes the $L_2(\mathcal T)-$norm.  

\medskip

\begin{proposition}\label{mean:L2-norm} 
Under the assumptions of Propositions \ref{variance-B}  and \ref{variance-V}, we have 
\begin{equation}\label{bterms}
\frac{1}{L}\sum_{j=0}^{L-1} \sum_{| \mathbf k |_1\leq 2(L+j)}\EE\left[(B_{ \mathbf k} - a_{ \mathbf k})^2 \right] \lesssim   C_\mu^2   L^D\overline M^{\; -1-2\alphamu/D} +    \|\mu \|_\infty ^2   L^{D+2}\overline M^{\; -1 - 2/D},
\end{equation}
 with the constant in $\lesssim$ depending only on $D, f_{\mathbf T}  $ and $\alpha_\mu$,   and
\begin{multline}\label{Vterms}
	\frac{1}{L} \sum_{j=0}^{L-1} \sum_{| \mathbf k|_1\leq2(L+j)}\EE[V_{\mathbf k}^2] = \left\{K_1 \frac{\varrho(D)(2^{2D+1} -2^D) L^D}{(D+1)!\overline M} 
	+K_2\sum_{i=1}^N \frac{ M_i(M_i-1)}{\overline M(\overline M - 1)} \right\}\times \{ 1+o(1)\},
\end{multline}
where $\varrho(D)$ is the constant in Proposition \ref{variance-V}, 
\begin{equation}
K_1=\int_{\mathcal T }\frac{ \{\sigma^{2}(\mathbf s)+ \gamma(\mathbf s,\mathbf s)\}}{f_{\mathbf T}(\mathbf s)}{\rm d}\mathbf s \quad \text{and} \quad  K_2=\int_{\mathcal T} \gamma(\mathbf s,\mathbf s){\rm d}\mathbf s.
\end{equation}
The constants in  the   $o(\cdot)$ quantity depend on $D$, $\sigma^2 $, $\gamma $, $f_{\mathbf T} $, but not on $\mu$ or $L$. 
\end{proposition}

\medskip

For $\mathbf t\in \mathcal T$, we can write 
\begin{equation}
 \widehat \mu(\mathbf t) - \mu(\mathbf t) 
 = \frac{1}{L} \sum_{j=0}^{L-1} \sum_{|\mathbf k|_1\leq 2(L+j)}\left\{  (B_{ \mathbf k} - a_{\mathbf k})  +   V_{\mathbf k}  \right\}  + \left[V_L(\mu; \mathbf{t})- \mu(\mathbf t) \right] .
\end{equation}
By Proposition \ref{mean:L2-norm}, we then deduce that, 
\begin{multline}\label{risk_bound_v}
	\EE\left\| \widehat   \mu -   \mu  \right\|_2^2 = \EE\left\| \widehat  V_L (\mu;\cdot)-  V_L (\mu;\cdot)  \right\|_2^2 +  \left\|V_L(\mu, \cdot)- \mu  \right\|_2^2 \\
	\leq  \left\{K_1 \frac{\varrho(D)(2^{2D+1}\!-2^D) L^D}{(D+1)!\overline M}  
	+ K_2\sum_{i=1}^N \frac{ M_i(M_i-1)}{\overline M(\overline M - 1)} \right\}\{1+o(1)\}	\\ +   \left\{  \frac{C_\mu}{\overline M^{\alphamu/D}} + \|\mu \|_\infty\frac{ L}{\overline M^{ 1/D}} \right\}^2  \frac{L^{D}}{\overline M} O(1) +  \left\|V_L(\mu, \cdot)- \mu  \right\|_2^2,
\end{multline}
and the constants in $O(1)$ and  $o(1)$ do not depend   on $\mu$ or  $L$.   With regard to the squared bias, by Proposition \ref{reste-control} we have   $\|V_L(\mu; \cdot)- \mu   \|_2^2\lesssim L^{-2\alphamu}$, provided $\mu\in\Sigma(\alphamu,C_\mu;D)\cap \mathcal F_D$. Thus the optimal choice of $L$ minimizing the  risk bound is 
\begin{equation}\label{L_general}
\hspace{-0.3cm} L^* \!= \!\left\lfloor C^* \overline M^{\; 1/(2\alphamu + D)}\right\rfloor
\quad \text{with}  \quad C^*=\left\{\frac{(2\alphamu \mathcal C_{\mu})(D+1)!}{DK_1 \varrho(D) (2^{2D+1}-2^D)} \right\}^{\! 1/(2\alphamu + D)}\!,
\end{equation}
and the corresponding optimal mean function estimator are
$ \widehat \mu^*(\cdot) = \widehat V_{L^*} (\mu; \cdot )$.

\medskip

\begin{corollary}\label{choice_L_v}
	If the conditions of Proposition \ref{mean:L2-norm} hold and  $\mu \in\Sigma(\alphamu,C_\mu;D)\cap\mathcal F_D$, then constants $C_1,C_2$ (not depending on $\mu$) exist such that, with the $o(1)$ from Proposition \ref{mean:L2-norm}, 
	\begin{equation}
		\EE\|\widehat \mu^*-\mu\|_2^2 \leq \left\{   \!\left[\mathfrak C_\mu \!+ \frac{C_1 C_\mu^2  }{\overline M^{\;\frac{2\alphamu}{D  }} } + \frac{C_2 \|\mu \|_\infty^2 }{\overline M^{\;\frac{4\alphamu}{D (2\alphamu+D)}} }\right] \!  \overline M^{\;\;-\frac{2\alphamu}{2\alphamu+D}} + K_2\sum_{i=1}^N \frac{ M_i(M_i-1)}{\overline M(\overline M - 1)}\right\}
		\times\{1+o(1)\} ,
	\end{equation}
	where $ \widehat \mu^*(\cdot) = \widehat V_{L^*} (\mu; \cdot )$ with $L^*$ in \eqref{L_general} and, with $\mathcal C_{\mu}$ the constant in Proposition \ref{reste-control}, 
	$$
	\mathfrak C_\mu=  \left\{\frac{2^{2D+1}-2^D}{(D+1)!}\right\}^{\frac{2\alphamu}{2\alphamu+D}}\left(K_1\varrho(D)\left\{\frac{2\alphamu \mathcal C_{\mu}}{ DK_1 \varrho(D)}\right\}^{\frac{D}{2\alphamu+D}}+C_{\mu} \left\{\frac{ DK_1 \varrho(D)}{2\alphamu \mathcal C_{\mu} }\right\}^{\frac{2\alphamu}{2\alphamu+D}}\right).
	$$
\end{corollary}

\medskip
 
\begin{remark}
Two regimes for the rate of the risk of the optimal  estimator $ \widehat \mu^*$ can be distinguished~: the \emph{sparse regime} [resp. \emph{dense regime}] corresponding to 
\begin{equation}\label{sparse_dense_reg}
\left(\frac{1}{\overline{M}}\right)^{\! -\; \frac{2\alphamu} {2\alphamu+D}}  \gg  \sum_{i=1}^N \frac{ M_i(M_i-1)}{\overline M(\overline M - 1)} \qquad \Bigg[\text{resp.  }
\left(\frac{1}{\overline{M}}\right)^{\! -\; \frac{2\alphamu} {2\alphamu+D}} \ll  \sum_{i=1}^N \frac{ M_i(M_i-1)}{\overline M(\overline M - 1)}\Bigg].
\end{equation}
It is worth noting that the sparse and dense regimes are not only determined by $M_1,\ldots,M_N$ and $N$, but also by the regularity of the mean function. 
\end{remark}

\medskip

\begin{remark}
	The term in \eqref{bterms} converges to zero faster than the term in \eqref{Vterms}. This very convenient fact is due to our method of estimating the Fourier coefficients. The constant $K_1$ is equal to the so-called \emph{the coefficient of difficulty} in the literature, see  \citet[Ch. 4]{Ef1999} for the case $D=1$, which allows to assess the complexity of a nonparametric regression problem. A remarkable aspect is that the fast convergence in \eqref{bterms} does not make $K_1$ depend on the $L_2-$norm of $\mu$, as is the case when the Fourier coefficients are estimated by simple sample means. This simplifies the choice of the optimal $L$. The construction of simple Fourier coefficients estimates, leading to a constant $K_1$ equal to the coefficient of difficulty, also seems to be a novelty in the problem of multivariate nonparametric heteroscedastic regression with random design, \emph{i.e.,} when $M_i=1$, $\forall  i $. 
\end{remark}


\begin{remark}
The rate of $\widehat \mu^*$ is optimal   over a large class of functions. More precisely, on the one hand, the estimation of the mean function can be formulated as an isotropic nonparametric regression problem, see \eqref{model-eq2}. By Corollary \ref{choice_L_v}, our $\widehat \mu^*$ achieves the minimax rate in the sparse regime   over increasing sub-classes of functions from $\Sigma(\alphamu,C_\mu;D)\cap\mathcal F_D$ with the uniform norm negligible compared to $\overline M^{\;2\alphamu/\{D (2\alphamu+D)\}}$.  Next, if $NM_i/\overline{M} \asymp 1$, the risk bound in the dense case has the parametric rate $N^{-1}$, the rate of the sample mean when all the trajectories $X_i$ are noiseless and observed at any domain point. See also \cite{CY2011}. 
\end{remark}


\section{Fourier coefficients inference}\label{sec:infer}

The following notation is used below~: for any $J\in\mathbb N$, let
\begin{equation}\label{J_bar}
	\overline J :=\operatorname{card}\{\mathbf k\in  \mathbb N^D: |\mathbf k|_1\leq  J\}\lesssim \max \{1,J^{D}\}.
\end{equation}
The constant in the last inequality depends only on $D$. Moreover, 
for any $J\in\NN$, $(v_{\mathbf k})^\top_{|\mathbf k|_1\leq J}\in \mathbb R^{\overline J}$ denotes the column matrix filled from top to bottom with the values $v_{\mathbf k}$, $\mathbf{k} = (k_1, k_2, \dots, k_D)$, considered in lexicographic order. 

We aim to construct inference for the pointwise error and the uniform error, that are
\begin{equation}\label{Deltas}
\Delta (\mathbf t;L) =  \widehat \mu(\mathbf t;L )- V_{L} (\mu; \mathbf t) \quad \text{ and } \quad
  \Delta_{\infty}(L)= 
\| \widehat \mu( \cdot;L)- V_{L} (\mu; \cdot) \|_\infty,  
\end{equation}
respectively, in particular with  $L=L^*$  derived in   \eqref{L_general}. Both   $\Delta (\cdot;L^*)$ and $\Delta_{\infty}(L^*)$   depend on the random vector $( \widehat a_{\mathbf k} - a_{\mathbf k} )_{|\mathbf k|_1\leq 4L^*-2 }^\top$, for which we will derive a non-asymptotic bound for the  Gaussian approximation. More generally, the control of the Gaussian approximation error for vectors of Fourier coefficient estimators also facilitates the approximation of the distribution of other quantities of interest, such as $\| \widehat \mu^{*} - V_{L^*}(\mu, \cdot)\|_2$. 

It is worth noting that 
$$
\mu^{*}(\mathbf t )- V_{L^*} (\mu;\mathbf t)= ( \widehat a_{\mathbf k} - a_{\mathbf k} )_{|\mathbf k|_1\leq 4L^*-2 }^\top \Phi_{L^*}(\mathbf t),
$$
where, for $\mathbf t\in \mathcal T$ and $J\in \mathbb N^*$,  
\begin{equation}\label{def:Phi}
	\Phi_J(\mathbf t) = \sum_{| \mathbf k|_1\leq2J}\phi_{\mathbf k}(\mathbf t)\mathbf u_{\mathbf k} + \sum_{\ell=1}^{J-1}\frac{J-\ell }{J} \left \{ \sum_{|\mathbf j|_1=2J+2\ell - 1}\phi_{\mathbf j}( \mathbf t)\mathbf u_{\mathbf j}+\sum_{|\mathbf j|_1=2J+2\ell}\phi_{\mathbf j}( \mathbf t)\mathbf u_{\mathbf j}\right \},
\end{equation}
and $ (\mathbf u_{\mathbf k})_{0\leq |\mathbf k|_1 \leq 4J-2} $ is the canonical basis of  $\mathbb R^{\overline{ 4J-2}}$.

\subsection{Berry-Esseen bounds by Stein's method}\label{BE_subsec}

Let $d\in \mathbb N$ and consider the centered random vector $(\widehat a_{\mathbf k }- a_{\mathbf k})^\top_{|\mathbf k|_1 \leq d}\in\mathbb R^{\overline d}$. Propositions \ref{variance-B} and \ref{variance-V} indicate that the leading term in the approximation of this vector is 
\begin{equation}\label{eq:def_W_Stein}
	W(d)=(V_{\mathbf k})_{|\mathbf k|_1\leq d}^\top\in\mathbb R^{\overline d},\qquad \text{with  } \quad V_{\mathbf k}= \sum_{(i,m)} \omega_{i,m} \frac{\eta_{i,m}}{f_{\mathbf T}(\mathbf T_{i,m})}\phi_{\mathbf k}(\mathbf T_{i,m}),
\end{equation}
 $\eta_{i,m}=\{\Xp{i} - \mu\} (\mathbf  T_{i,m}) + \sigma (\mathbf  T_{i,m}) e_{i,m} $, and $\overline d$ defined according to \eqref{J_bar}.
It worth recalling that $V_{\mathbf k}$ is the variance term in the decomposition of $(\widehat a_{\mathbf k }- a_{\mathbf k})$ in \eqref{estimator:decom}. Thanks to the control neighbors approach we propose, the bias term $(B_{\mathbf k }- a_{\mathbf k})$ is negligible and the distribution of $W(d)$  does not depend on the mean function which has to be estimated.

Let us point out that the components of $W(d)$ have an intricate dependence structure due to the weights $\omega_{i,m}$ and the variables $\eta_{i,m}$. In order to derive Gaussian approximation error bounds, we adopt Stein's method. More precisely, we follow the exchangeable pair approach in the multivariate setting, as developed by  \cite{MVNAReinert2009}; see also  \cite{Chatterjee2008}.

\medskip

\begin{theorem}\label{normal_approximation}
Assume that (H\ref{assuH2}) to  (H\ref{assuH3})  and (D\ref{assuD1}) to (D\ref{assuD5}) hold true, and $\EE(|e|^3)<\infty$. Moreover, for $ \mathbf t, \mathbf s \in \mathcal T$,  assume that 
$m_4(\mathbf t,\mathbf s)= \EE\left[\{X-\mu\}^2(\mathbf t)\{X-\mu\}^2(\mathbf s)  \right]$
exists and the map $(\mathbf s, \mathbf t) \mapsto m_4(\mathbf t,\mathbf s)$ is integrable.  Then, if $Z$ is a $\overline d$-dimensional standard Gaussian vector, 
for any three-times differentiable function $h : \mathbb R^{\overline d} \rightarrow \mathbb R$,
\begin{equation}
\left|\EE[ h( W(d))] - \EE[ h(\Sigma^{1/2} Z)]\right| \leq  \mathfrak F  \overline d \left[\frac{|h|_3 }{ \overline M^{\;2}}+|h|_2\left\{ \frac{1 }{  \overline M^{\;3}} + \frac{ 1 }{ \overline M^{\;4}} \sum_{i =1}^N  M_i^2\right\}^{1/2}\right],
\end{equation}
where $ \Sigma=\EE[ W(d)W(d)^\top] $, $\mathfrak F$ is a constant which does not depend  on $h$, and
$$
|h|_2:= \sup_{p,q}\left\| \frac{\partial^2 h}{\partial x_p \partial x_q} \right\|_\infty, \qquad |h|_3:= \sup_{p,q,r}\left\| \frac{\partial^3 h}{\partial x_p \partial x_q \partial x_r} \right\|_\infty.
$$
\end{theorem}

\medskip

The constant $\mathfrak F$ in Theorem \ref{normal_approximation} is determined  by the terms $\mathfrak a_1$, $\mathfrak a_2$ and $\mathfrak c_1$ in the proof. The covariance matrix $ \Sigma= \mathbb E\left[ W(d)W(d)^\top \right]= (\Sigma_{\mathbf p,\mathbf q})_{|\mathbf p|_1,|\mathbf q|_1\leq d}$ has the 
components 
\begin{multline}\label{comp_Sig}
	\Sigma_{\mathbf p,\mathbf q}=\frac{\varrho(D)}{\overline M} \int_{\mathcal T}\frac{\sigma^2(\mathbf t) + \gamma (\mathbf t, \mathbf t) }{f_{\mathbf T}(\mathbf t)} \phi_{\mathbf p}(\mathbf t) \phi_{\mathbf q}(\mathbf t) {\rm d}\mathbf t \times \{1 + o(1)\}\\
	+ \frac{1}{\overline M(\overline M-1)} \sum_{i=1}^N M_i(M_i-1) \int_{\mathcal T^2} \gamma(\mathbf t, \mathbf s) \phi_{\mathbf p}(\mathbf t) \phi_{\mathbf q}(\mathbf s){\rm d}\mathbf t {\rm d}\mathbf s \times \{1+o(1)\}.
\end{multline}

\medskip

Theorem \ref{normal_approximation} and Proposition \ref{limit_law} allow  to derive the normal approximation for quantities that depend on Fourier coefficients, such as any linear combination. We  define the  normalizing rates~: 
 \begin{equation}\label{eq:rate_V}
	r_1(L)=\sqrt{\frac{\overline M}{L^D} }, \qquad \text{and} \qquad 
	r_2= \frac{\overline M}{ \sqrt{\sum_{1\leq i \leq N} M_i(M_i-1)}} .
\end{equation}


\begin{corollary}\label{coro_linear}
The conditions of Theorem \eqref{normal_approximation} are met, and $\min \{r_1(L^*), r_2\}\rightarrow \infty$ with $L^*$ as in \eqref{L_general}. Then, for any vector $w$  with the same dimension as  $W(4L^*-2)$, the variables
$$
\min \{r_1(L^*), r_2\}W(4L^*-2)^\top w \qquad \text{ and  } \qquad \min\{r_1(L^*)
, r_2\}( \Sigma^{1/2}Z)^\top w
$$ 
have the same  Gaussian asymptotic distribution, provided the components of $w$ are bounded.
\end{corollary}
 
\medskip

Corollary \ref{coro_linear} is obtained by applying Theorem \ref{normal_approximation}  with the sub-class of  three-times differentiable functions $h$ defined as $h( W(4L^* -2)) = g(W(4L^* -2)^\top w)$ with $g: \mathbb R\rightarrow \mathbb R$. Choosing $w= \Phi_{L^*}(\mathbf t)$,    Corollary \ref{coro_linear} implies that  $\min\{r_1(L^*), r_2\}\Delta(\mathbf t;L^*)$ satisfies a Central Limit Theorem, for any  $\mathbf t\in \mathcal T$, provided the variance of $\min\{r_1(L^*)
, r_2\}( \Sigma^{1/2}Z)^\top \Phi_{L^*}(\mathbf t)$ has a finite limit. This latter fact is established in the following lemma.

\begin{lemma}\label{limit_law}
The conditions of Theorem \eqref{normal_approximation} are met, and $\min \{r_1(L^*), r_2\}\rightarrow \infty$. Then
\begin{multline}\label{var_asymp}
	(\min \{r_1(L^*), r_2\})^2 \Phi_{L^*}(\mathbf t)^\top \Sigma \Phi_{L^*}(\mathbf t) = \{1+o(1)\} \frac{(\min \{r_1(L^*), r_2\})^2}{r_2^2(L^*)}\times \gamma(\mathbf t,\mathbf t)\\
	+\{1+o(1)\} \frac{(\min \{r_1(L^*), r_2\})^2}{r_1^2(L^*)} \times \frac{\varrho(D)(2^{2D+1}-2^D)}{(D+1)!}\times  \frac{\sigma^2(\mathbf t) + \gamma (\mathbf t, \mathbf t) }{f_{\mathbf T}(\mathbf t)}.
\end{multline} 

\end{lemma}

\medskip

The right-hand side of \eqref{var_asymp} reveals the two types of asymptotic variances, corresponding to the cases $r_1(L^*)\ll r_2$ (sparse regime) and $r_1(L^*)\gg r_2$ (dense regime), respectively. See also \cite{ZW2016} for similar findings in the case $D=1$, for twice differentiable mean functions, and with kernel smoothing. As a consequence of Theorem \ref{normal_approximation} and Lemma \ref{limit_law}, for the supremum distance we prove the following result, which appear to be the first of this kind covering both sparse and dense regimes.

\medskip

\begin{corollary}\label{coro_CB}
The conditions of Corollary \ref{coro_linear} hold true. 
Then,  
$$
\sup_{t\in\RR} \left|\PP \left(\min\{r_1(L^*),r_2\} \Delta_\infty(L^*)\leq t \right)\!- \PP \left(\min \{r_1(L^*),r_2\} 
\left\| (\Sigma^{1/2}Z)^\top \Phi_{L^*}(\cdot) \right\|_\infty\!\leq t \right) \right|\rightarrow 0,
$$  
and the convergence holds uniformly over  $\{\mu\in\Sigma(\alphamu,C_\mu;D): \|\mu\|_\infty\leq \psi(\overline M) \}$ for any   $\psi(\overline M)\rightarrow \infty $ such that $ \psi(\overline M)/\;  \overline M^{\; 2\alphamu/\{D (2\alphamu+D)\}}\rightarrow 0$.
\end{corollary}

\medskip

Corollaries \ref{coro_linear} and \ref{coro_CB} allow to construct asymptotic confidence regions centered at  \emph{de La Vallée Poussin} sum $V_{L^*}(\mu, \cdot)$. We next provide the theoretical grounds for the construction of asymptotic confidence regions centered at the mean function, provided $V_{L^*}(\mu, \cdot)$ suitably approximates $\mu$, as is the case for $\mu \!\in \!\mathcal F_D\!=\overline{\operatorname{Span}}\left\{ \phi_{\mathbf{k}}: \mathbf{k}\in \mathbb N ^D\right\} $.

\medskip

\begin{corollary}\label{coro_CB2}
	Suppose the conditions of Corollary \ref{coro_linear} hold true and $\mu\in\Sigma(\alphamu,C_\mu;D)\cap \mathcal F_D$. For any $\vartheta(\overline M)\rightarrow \infty $ such that, say, $\log(\overline{M})/\vartheta(\overline M)\rightarrow\infty$, and 
	$$
	L^\diamondsuit = \vartheta(\overline M) \overline{M}^{\; \frac{1}{2\alpha_\mu + D}} \asymp \vartheta(\overline M)L^*, 
	$$
 for any $\mathbf t \in\mathcal T$,  	the quantities 
	$$
	\min\{r_1(L^\diamondsuit), r_2\} \left( \widehat{\mu}(\mathbf{t},L^\diamondsuit ) -  \mu(\mathbf{t}) \right) \quad \text{and} \quad \min\{r_1(L^\diamondsuit), r_2\} (\Sigma^{1/2}Z)^\top \Phi_{L^\diamondsuit}(\mathbf t),
	$$
	have the same asymptotic distribution. Moreover,  
	$$
	\min\{r_1(L^
	\diamondsuit),r_2\}
   \| \widehat \mu( \cdot;L^\diamondsuit)-  \mu(\cdot) \|_\infty  \quad \text{  and } \quad 
		\min \{r_1(L^\diamondsuit),r_2\} 
 	\left\| (\Sigma^{1/2}Z)^\top \Phi_{L^\diamondsuit}(\cdot) \right\|_\infty , 
	$$
		have the same asymptotic distribution.
\end{corollary}


\begin{remark}
Considering the expression \eqref{comp_Sig} of the components of $\Sigma$, it is worth noting that the Gaussian approximation given by the Corollary \ref{coro_CB2} does not depend on the mean function, which is the target of the inference. This pivotality property is due to our approach to the Fourier coefficients estimation using fast integration. 
\end{remark}


\begin{remark} Corollary \ref{coro_CB2} holds for a sequence $\vartheta(\overline M) $ which increases to infinity at an arbitrarily slow rate. The conclusion of Corollary \ref{coro_CB2} is also valid with $\vartheta(\overline M)\asymp 1$  as soon as $r_2/r_1(L^*) \rightarrow 0$, that is in the dense regime. 
\end{remark}

In this section we have assumed that the matrix $\Sigma$ with the entries in \eqref{comp_Sig} is given. In practice, the entries can be estimated by replacing $\sigma^2(\mathbf t)$ and $\gamma(\mathbf t,\mathbf t)$ with uniformly consistent estimators.
In the Supplement we  propose an alternative method for constructing the confidence regions based on subsampling and without using  the matrix $\Sigma$.


\section{Regularity estimation}\label{sec:reg}

Estimation and inference have been constructed above assuming that the Hölder exponent $\alphamu$ is given in Assumption (H\ref{assuH3}). We now want to provide an estimate of this exponent and make our approach adaptive. However, it is well known that a Hölder class is too large to allow identification of the exponent $\alphamu$. It includes functions with any number of partial derivatives, making it impossible to identify the exponent $\alphamu$. See, for example, \cite{HN2011} and \cite{PR2019} for discussions on closely related problems. It is for this reason that the notion of \emph{pointwise Hölder exponent}, as borrowed from mathematical analysis, is considered. This notion defines the pointwise regularity around a point in the domain.  We then introduce a rich class of functions with a constant pointwise Hölder exponent for which this exponent can be identified and consistently estimated. Below, $B(\mathbf{t}, \eta)\subset \mathbb R^D$ denotes the ball centered at  $\mathbf t$ with radius $\eta$.

\medskip

\begin{definition}\label{pointwiseholder}
Let $s>0$   and $\mathbf{t}_0 \in \mathcal T$. A function $f: \mathcal{T} \rightarrow \mathbb{R}$ belongs to the class $C^s_{\mathbf{t}_0}$ if and only if   a radius $\eta > 0$, a polynomial $P$, and a constant $C$ exist such that 
\begin{equation}\label{cond_a}
    \forall \mathbf{t} \in B(\mathbf{t}_0, \eta), \qquad |f(\mathbf{t}) - P(\mathbf{t} - \mathbf{t}_0)| \leq C \|\mathbf{t} - \mathbf{t}_0\|^s.
\end{equation}
The pointwise Hölder exponent is defined by $\nu_f(\mathbf{t}_0) = \sup \{ s : f \in C^s_{\mathbf{t}_0} \}$.
\end{definition}

\medskip

The constant $C$ in \eqref{cond_a} can depend on $f$ and $s$.  Note that when the function $f$ is not differentiable, condition \eqref{cond_a} becomes~: 
    $\forall \mathbf{t} \in B(\mathbf{t}_0, \eta)$, 
    $|f(\mathbf{t}) - f(\mathbf{t}_0)| \leq C \|\mathbf{t} - \mathbf{t}_0\|^s$.
The reader may refer to \cite{guiheneuf19982}, \cite{jaffard1991pointwise}, \cite{S2016}   for more details on the pointwise Hölder exponent.

\medskip

\begin{remark}\label{exp_Sig9} Let us provide some examples and insight on the pointwise Hölder exponent notion. 
For example,  the function $\mathbf{t} \mapsto \|\mathbf{t}\|^{\nu}$, where $\nu$ is a positive number $(\nu \notin 2\mathbb{N})$, has a pointwise Hölder exponent at the origin equal to $\nu$. 
Furthermore,  for $\nu$ and $\beta$ positive numbers, the function $\mathbf{t} \mapsto \|\mathbf{t}\|^{\nu} \sin\left( \|\mathbf{t}\|^{-\beta} \right)$, extended by continuity at the origin,   also has a Hölder pointwise exponent at the origin equal to $\nu$.  Finally, note that the sample paths of the Brownian motion, considered on a compact interval $\mathcal T$, are almost surely not Hölder continuous of order 1/2, and Hölder continuous of any order less than 1/2; see \citep[Theorem 2.2, Corollary 2.6]{Yor1999}. This means that with probability 1 the  pointwise Hölder exponent of the Brownian motion is \emph{constant} and equal to 1/2. The example of the Brownian motion sample paths reveals that the exponent $\alphamu$ in Assumption (H\ref{assuH3}) can be arbitrarily close to the pointwise Hölder exponent if the latter is constant.
\end{remark}

\begin{remark}\label{Rem10}
It is worth discussing on  pointwise Hölder exponents greater than 1. For simplicity, consider $\mathbf t\in\mathbb R$. On the one hand, it can be shown  that the pointwise Hölder exponent at 0 of any primitive of the function $\mathbf{t} \mapsto |\mathbf{t}|^{\nu}$ is equal to $\nu + 1$. On the other hand, any primitive of the function $\mathbf{t} \mapsto |\mathbf{t}|^{\nu} \sin\left( |\mathbf{t}|^{-\beta} \right)$ has the pointwise Hölder exponent at 0 equal to $\nu + 1+\beta $. The difference between the two examples is due to the presence of oscillations around 0, which after integration increase the pointwise Hölder exponent.
 \end{remark}
 
 \medskip

In this section, we focus on the functions with a constant pointwise Hölder exponent. 
It worth noting that the class of functions  with  constant pointwise Hölder exponent
is rich. Indeed, for any $s>0$, there exists a dense open set in the space of Hölder functions of exponent $s>0$ defined on $\mathcal T$, denoted by $C^s (\mathcal T)$, such that any function in this set has a constant pointwise Hölder exponent equal to $s$. See \citet[Theorem 2.9]{S2016}. Here, we focus on the case where the constant pointwise Hölder exponent belongs to  $(0,1)$. Remark \ref{Rem10} reveals that the case of  pointwise Hölder exponent larger than 1 is even more challenging. This case is left for future study. 

\medskip

\begin{definition}\label{pointwiseSigtilde}
For $\alpha \in (0,1)$ and $r\geq 0$, let $\widetilde{C}^\alpha (\mathcal T;r)$ be the class functions $f:\mathcal T \rightarrow \RR$ with constant pointwise Hölder exponent   $\nu_f(\mathbf{t})=\alpha$,  $\forall \mathbf t \in\mathcal T$, for which a constant $C_f$ exists such that    $\forall \epsilon >0$ it holds
$$
\sup_{\mathbf t_1\neq  \mathbf t_2}\frac{|f(\mathbf t_1)- f(\mathbf t_2)|}{\|\mathbf t_1- \mathbf t_2\|^{\alpha-\epsilon}}\leq \epsilon^{-r} C_f.
$$
\end{definition}

\medskip

Note that $\widetilde{C}^\alpha (\mathcal T;r^\prime ) \subset \widetilde{C}^\alpha (\mathcal T;r)$ for any $0\leq r^\prime \leq  r$. Moreover, if $f\in\widetilde{C}^\alpha(\mathcal T;r)$, then  $\forall s < \alpha$,   $\exists C_{f,s} > 0$ such that 
\begin{equation}\label{cond}
	|f(\mathbf{t}_1) - f(\mathbf{t}_2)| \leq C_{f,s} \|\mathbf{t}_1 - \mathbf{t}_2\|^s, \quad \forall \mathbf{t}_1, \mathbf{t}_2 \in \mathcal{T}, \quad \text{ with } \; C_{f,s}=C_f |\alpha-s|^{-r}.
\end{equation}
By the Levy's modulus continuity for the fractional Brownian motion, see \citep[page 220]{L1995}, we have that with probability 1, the sample paths of a fractional Brownian motion with Hurst exponent $H\in(0,1)$ belong to  $\widetilde{C}^H([0,1];r)$, for any $r\geq 1/2$.   For refined results on the existence and the construction of functions with prescribed Hölder exponent, including the Weierstrass function, see \cite{DVM1998}. 
By the definitions and \citet[Theorem 2.9]{S2016}, for any $r\geq 0$, the set $\widetilde{C}^\alpha(\mathcal T;r)\cap C^\alpha (\mathcal T)$ is dense in $C^\alpha (\mathcal T)$. This shows that our classes $\widetilde{C}^\alpha(\mathcal T;r)$ are rich.

There has been some recent interest for the regularity estimation with functional data, particularly when regularity is characterized through an equality or near-equality, as is the case for the Hurst index of a fractional Brownian motion (e.g., \cite{golo1, Golovkine2021, wang2023adaptive,  kassi2023, kassi2024structural}). See also \cite{hsing2020}. However, when the regularity is defined through an inequality as in \eqref{cond}, the regularity estimation  is more challenging. To address the problem, let us consider that $\mu$ belongs to a class $\widetilde{C}^\alpha(\mathcal T;r)$ for some $\alpha \in (0,1)$ and $r\geq 0$. By definition, if Assumption (H\ref{assuH3}) holds then necessarily $\alphamu \leq \alpha$.  If $\mu$ has a constant pointwise Hölder exponenent equal to $\alphamu$ from Assumption (H\ref{assuH3}), then $\mu\in \widetilde{C}^{\alphamu} (\mathcal T;0)$.

Let $\mathcal{K}$ be a non-negative integer that will increase with the sample size,   and define
\begin{equation}
I_k = \left[\frac{k}{\mathcal{K}}, \frac{k+1}{\mathcal{K}}\right) \quad \text{and} \quad t_k = \frac{k+1/2}{\mathcal{K}} \qquad \text{for } k \in \{0, \dots, \mathcal{K}-1\}.
\end{equation}
For any vector $\mathbf{k} = (k_1, k_2, \dots, k_D)$ with integer components  in $\{0, \dots, \mathcal{K}-1\}$, define
\begin{equation}\label{def:pnt_intv}
I_{\mathbf{k}} = I_{k_1} \times I_{k_2} \times \dots \times I_{k_D}  \qquad \text{ and } \qquad \mathbf{t}_{\mathbf{k}} = (t_{k_1}, t_{k_2}, \dots, t_{k_D})\in I_{\mathbf{k}}.
\end{equation}
Finally, let
 $$
 b_{\mathbf k}=\{\sum_{\mathbf k^+} 1\}^{-1}\sum_{\mathbf k^+}\{\mu(\mathbf t_{\mathbf k}) - \mu(\mathbf t_{\mathbf k^+})\}^2,\quad \text{ for each } \quad  \mathbf k\in \mathcal X(\mathcal K),
 $$ 
with  
$$
  \mathcal X(\mathcal K):= \{\mathbf j \in \mathbb N^D:   | \mathbf j|_1\leq D( \mathcal{K} - 1)-1, |\mathbf j|_\infty \leq \mathcal K -1\} ,
$$
where $\mathbf k^+$ is any $D$-dimensional   vector with $|\mathbf k^+|_1= |\mathbf k|_1+1$ and integer components   in $\{0, \dots, \mathcal{K}-1\}$. The quantities $b_{\mathbf k}$ will be used to estimate  $\alpha_\mu$.  Note that the cardinality of the set $ \mathcal X(\mathcal K)$ is equal to $\mathcal K^D -1$.

If $\mu \in \widetilde C^{\alpha} (\mathcal T;r)$, the quantity $ b_{\mathbf k}$  has the rate  $O(\mathcal{K}^{-2s})$ for any $s < \alpha$, and, for any $\epsilon >0$, $ b_{\mathbf k}\mathcal{K}^{2(\alpha+\epsilon)} $ increases to infinity with $\mathcal{K}$. The idea is then to derive an estimate of $\alpha$ from the logarithm of an estimate of $ b_{\mathbf k}$. 
Let $\mathcal J$ be a large integer and define 
  \begin{equation}\label{def_HJ}
  	\mathcal H= \mathcal H_{\mathcal J}=  \left\{ H_j= j/\mathcal J: 1\leq j\leq \mathcal J \right\}.
  \end{equation}
Then, for any $	\mathcal H\ni H_j < \alpha$, it holds that $b_{ \mathbf k} + \mathcal{K}^{-2H_j} \sim\mathcal{K}^{-2H_j}$. This allows to derive an estimating equation for $H_j$ from $b_{\mathbf k} + \mathcal{K}^{-2H_j}$, as long as $H_j < \alpha$. In other words, for any small $\epsilon >0$, we have
\begin{equation}
 -\frac{1}{2\log(\mathcal{K})} \log\left(b_{\mathbf k} + \mathcal{K}^{-2H_j}\right) =H_j + O\left(\epsilon^{-r}\mathcal{K}^{-2(\alpha - H_j)+\epsilon}\right)=H_j \{1+o(1)\}, \quad \text{ for large } \mathcal K.
\end{equation}
Meanwhile, this estimating equation breaks down if $H_j > \alpha$, in the sense that the left side is less than the right side. Our estimation idea is then the following~:   estimate $b_{\mathbf k}+ \mathcal{K}^{-2H_j}$ from the data,  incrementally increase $H_j$ in the range $\mathcal H$, and retain the first $H_j$ for which $-\{2\log(\mathcal{K})\}^{-1} \log\left(b_{\mathbf k} + \mathcal{K}^{-2H_j}\right)$ is sufficiently far from $H_j$. 
Such a breaking point indicates that $H_{j-1}$ and $H_j$ are close to  $\alpha$. Given the choice of $\mathcal H_{\mathcal J}$, our procedure boils down to construct an estimator for the integer $j_0= \lfloor \mathcal J\times\alpha \rfloor$ and divide it by $\mathcal J$ to get the estimator of $\alpha$. Before providing the details of our procedure,  note that the definition of $\widetilde C^{\alpha} (\mathcal T;r)$ allows to consider an average of the $b_{\mathbf k}$ with respect to $\mathbf k$ in the estimating equation to improve the performance of the procedure. Finally, our estimation strategy requires a concentration bound for the estimator of $H_j$ obtained through the estimating equation when $H_j< \alpha$. We elaborate on these ideas in the following.

Let  $\mathbf k \in \mathbb N^{D}$  and, with the rule $0/0=0$,  define 
$$
F_{\mathbf k}(\mathbf T_{i,m}) = \frac{\mathbf{1}_{\{\mathbf T_{i,m} \in I_{\mathbf k}\}} }{\sum _{(j, m^\prime)} \mathbf{1}_{\{\mathbf T_{j,m^\prime} \in I_{\mathbf k}\}}},
$$
where $I_{\mathbf k}$ is the hypercube introduced in \eqref{def:pnt_intv}. Here, $F_{\mathbf k}(\mathbf T_{i,m})$ is weighting factor, accounting for the data points in  $I_{\mathbf k}$. We estimate $b_{\mathbf k}$ by
\begin{equation}\label{def_b_kgras}
\widehat b_{\mathbf k} = \frac{1}{\sum_{\mathbf k^+} 1}\sum_{\mathbf k^+}\Big\{ \sum _{(i, m)}  Y_{i,m} \{F_{\mathbf k}(\mathbf T_{i,m}) -  F_{\mathbf k^+}(\mathbf T_{i,m}) \}\Big\}^2.
\end{equation}
This estimator  leverages the differences between adjacent weighted hypercubes and captures the squared increments in the mean function $\mu$. 
 
\medskip

\begin{proposition}\label{pas_encore}
The Assumptions   (H\ref{assuH2}), (H\ref{assuH12}) and (D\ref{assuD1}) to (D\ref{assuD5}) hold true. Moreover, $\mu \in \widetilde C^{\alpha} (\mathcal T;r)$, for some $\alpha \in (0,1)$. 
	Then, for a fixed index $\mathbf k\in  \mathcal X(\mathcal K)$, and $\mathbf T_1, \mathbf T_2 $ independent copies of $\mathbf T$,  we have
	\begin{equation}
	\EE\left[\widehat b_{\mathbf k}\right] = \frac{1}{\sum_{\mathbf k^+} 1}\sum_{\mathbf k^+}\EE\left[\left\{\mu(\mathbf T_1) - \mu(\mathbf T_2)\right\}\frac{\mathbbm{1}\{ (\mathbf T_1, \mathbf T_2) \in I_{\mathbf k} \times I_{\mathbf k^+}\} }{p_{\mathbf k} p_{\mathbf k^+}} \right]^2 + \mathcal{G},
	\end{equation}
	where $p_{\mathbf k} = \mathbb{P}(\mathbf T \in I_{\mathbf k})$, $p_{\mathbf k^+} = \mathbb{P}(\mathbf T \in I_{\mathbf k^+})$. The remainder term $\mathcal{G}$ satisfies
	\begin{multline}
\left|	\mathcal{G}\right|\leq   \frac{\mathcal{K}^D}{C_0 (\overline{M} + 1)} G_{\mathbf k} \\ + \frac{\sum_{i=1}^N M_i(M_i-1)}{\overline{M} (\overline{M} - 1)} \times \frac{1}{\sum_{\mathbf k^+} 1} \sum_{\mathbf k^+}  \Big\{\gamma(\mathbf t_{\mathbf k},\mathbf  t_{\mathbf k}) +  \gamma(\mathbf t_{\mathbf k^+},\mathbf t_{\mathbf k^+}) -2 \gamma(\mathbf t_{\mathbf k},\mathbf  t_{\mathbf k^+}) \Big\},
	\end{multline}
	with
	$$
	G_{\mathbf k} = \mu^2(\mathbf t_{\mathbf k}) + \sigma^2(\mathbf t_{\mathbf k}) +   \gamma(\mathbf t_{\mathbf k}, \mathbf t_{\mathbf k}) +\frac{1}{\sum_{\mathbf k^+} 1}\sum_{\mathbf k^+}\{ \mu^2(\mathbf t_{\mathbf k^+}) + \sigma^2(\mathbf t_{\mathbf k^+})+\gamma(\mathbf  t_{\mathbf k^+}, \mathbf  t_{\mathbf k^+})\}.
	$$
\end{proposition} 

\medskip

If $\mu \in \widetilde C^{\alpha} (\mathcal T;r)$, the quantity $\EE[\widehat b_{\mathbf k}] - \mathcal G$ 
appearing in Proposition \ref{pas_encore}, is expected 
to have the  rate  $O(\mathcal{K}^{-2s})$,  $\forall s < \alpha$, and  not decrease faster than $\mathcal{K}^{-2(\alpha+\epsilon)}$, for any $\epsilon >0$.
Let 
\begin{equation}\label{g_cal}
	g(\mathcal{K}) = (\mathcal{K}^D-1)^{-1} \sum_{\mathbf k \in \mathcal X(\mathcal K)} \EE\big[\; \widehat{b}_{\mathbf k}\big].
\end{equation}
For $1\leq j\leq \mathcal J $ such that $H_j < \alpha$, let
\begin{equation}\label{proxy_Hj}
\widetilde H_j =-\frac{1}{2} \log(\mathcal K )^{-1}\log\left(	g(\mathcal{K}) + \mathcal  K^{-2H_j}\right),
\end{equation}
and the  estimator
\begin{equation}\label{estimator_Hj}
\widehat H_j =-\frac{1}{2} \log(\mathcal K )^{-1}\log\Bigg(\frac{1}{\mathcal K^D-1} \sum_{\mathbf k \in \mathcal X(\mathcal K)}\widehat b_{\mathbf k}+ \mathcal  K^{-2H_j}\Bigg).
\end{equation}
To understand the role of $\widetilde H_j $, using the inequality $\log(1+y)\leq y$, $\forall y\geq 0$, we get
    \begin{equation}\label{bias_Hj}
	|H_j - \widetilde{H}_j| \leq\frac{1}{2\log(\mathcal K)} \mathcal K^{2H_j}g(\mathcal{K}),
\end{equation}
with $g(\mathcal{K})$ in \eqref{g_cal}.
 Since $\mu \in\widetilde C^{\alpha}(\mathcal T;r)$,   for any $0<s<\alpha$, by Proposition \ref{pas_encore} we have
\begin{equation}\label{qaz35}
|H_j - \widetilde{H}_j| \lesssim\frac{1}{\log(\mathcal K)}\left\{ \mathcal K^{2H_j} \mathcal N^{-1} + \mathcal K^{D+2H_j}(\mathcal N \mathfrak m)^{\; -1} + (\alpha - s)^{-r} \mathcal K^{2(H_j- s)}\right\}, 
\end{equation}
where $\mathcal N^{-1}= \overline M ^{\;-2}\sum_{i=1}^N  M_i^2$ and  $\mathfrak m= \overline M ^{\;-1} \sum_{i=1}^N  M_i^2 $. Thus, $ \widetilde{H}_j$ is a non-random proxy for $H_j$ which approaches the target if $\mathcal K$ increases at a suitable rate depending on sample size, provided that $H_j< \alpha$. Meanwhile, $ \widetilde{H}_j$ stays away from $H_j$ if $H_j > \alpha$. In conclusion,  $ \widetilde{H}_j$ can be used to identify a small neighborhood of $\alpha$. The remaining part is the control of the difference $\widehat H_j - \widetilde{H}_j$, which will be done using the following subgaussianity assumption. 

\medskip

\begin{assumptionD}
	
	\item\label{assuD6} Positive constants  $\mathfrak A$ and $\mathfrak B$ exist such that $\forall p>1$
	$$
	\sup_{\mathbf t\in \mathcal T}\EE\left[\{X(\mathbf t)-\mu(\mathbf t)\}^{p} \right]\leq \mathfrak A^{p} p^{\frac{p}{2}},\quad \EE|\varepsilon| ^p\leq \mathfrak B^{p} p^{\frac{p}{2}}.
	$$ 
	
		\item\label{zerport_main} Constants $\underline C_M,\overline C_M>0$ exist such that $\underline C_M \mathfrak m\leq M_i \leq \overline C_M \mathfrak m$, $ \forall 1\leq i \leq N$.	
		
\end{assumptionD}

\medskip

\begin{proposition}\label{pour_demain}
    Assume that conditions (H\ref{assuH2}), (H\ref{assuH12}) and (D\ref{assuD1}) to (D\ref{zerport_main}) are satisfied.   
    Then, for any $\eta>0$, 
\begin{multline}\label{zertnada}
\mathbb P \left(|\widehat H_j - \widetilde H_j|\geq \eta\right)\leq 2\exp \left ( -  \frac{\mathfrak k_1^2 (\mathcal K^{D}-1) \mathcal N}{ \log^2(N)+\mathfrak k_1^2\mathcal N g(\mathcal K)}\times \frac{2\eta^2 g(\mathcal K)\log^2(\mathcal K) \mathcal K^{-4\eta}}{ 1 + \eta \log(\mathcal K)\mathcal K^{-2\eta}  } \right)\\
+\frac{2}{N}+ 2 N\exp\left(- 2C_0^2\sqrt{\underline C_M}\mathcal K^{-2D}\mathfrak m^{1/2}\right).
\end{multline}
where $\mathfrak k$ is a positive constant and  $g(\mathcal{K})$ defined as in \eqref{g_cal}. 
\end{proposition}

\medskip

The last term in the bound \eqref{zertnada} converge to zero if $\mathcal K^{2D} \mathfrak m^{-1/2}= o(\log(N))$.
The optimal choice of $\mathcal K$ depends on the unknown regularity of the mean function. To avoid a two-step procedure, it suffices to choose  $ \mathcal K_\tau= \exp( \log(N) ^\tau)$ with $\tau \in (0,1)$. It is worth noting that   the condition $ \mathcal K_\tau^{2D} \mathfrak m^{-1/2}= o(\log(N))$ is implied by $ \log(N)^\tau \lesssim \log(\mathfrak m)$.
In terms of asymptotic behavior,  $\mathcal K_\tau$ lies between any positive power of $ \log(N)$ and any positive power of $N$.
Moreover, by assumption (D\ref{zerport_main}), we have $ \underline C_M / \overline{C}_M\leq \mathcal N/ N \leq \overline C_M/\underline {C}_M$. Thus, for any $0<s<\alpha$, by \eqref{qaz35} and \eqref{zertnada}, we get
\begin{align}
 \quad  H_j - \widetilde H_j = & \; O\Big( \log(N)^{-\tau} \left\{\mathcal K_\tau^{2(H_j-s)} (\alpha-s)^{-r}
 + \mathcal N^{-1} \mathcal K_\tau^{2H_j}
 \right\} \Big),\\
 \widehat H_j - \widetilde H_j= & \; O_{\mathbb P}\left ( \{ \log(N)^{2-2\tau}/( g(\mathcal K_\tau) N\mathcal K_\tau^{D}) + 1/( \mathcal K_\tau^{D} \log(N)^\tau)\}^{1/2}
 \right).\label{bias_variance_Hj}
\end{align} 
Therefore, for $H_j< \alpha$, we get
$$
 \widehat H_j - \widetilde H_j= o_{\mathbb P}(1),\qquad   H_j - \widetilde H_j =o(1),\qquad \text{ and thus  } \qquad
\widehat H_j -  H_j= o_{\mathbb P}(1) ,
$$ 
provided $ H_j < \alpha$. Note that, for $\tau^\prime  \in (0,\tau)$, by the rates in \eqref{bias_variance_Hj}, we have
\begin{equation}\label{guarantee_J}
\widehat H_j - H_j = o_{\PP}(\log(N)^{-\tau^\prime}),
\end{equation}
provided $ H_j < \alpha$. Therefore, our estimator of $j_0= \lfloor \mathcal J\times\alpha \rfloor$ is 
\begin{equation}\label{j0_hat45}
 \widehat j_0 = \max \left\{ j: |\widehat H_j - H_j| \leq  \log(N)^{-\tau^\prime}\right\},\qquad \text{for some } \tau^\prime \in (0,1). 
\end{equation}
 By definition, $ 0\leq \alpha - j_0/\mathcal J < \mathcal J^{-1}$. Next, we provide a concentration bound for  $\widehat j_0/ \mathcal J$.

\medskip

\begin{proposition}\label{pour_auj}
	Suppose the conditions of Proposition \ref{pour_demain} hold true and $\mu\in\widetilde C^{\alpha} (\mathcal T;r)$, for some $r\geq 0$. Consider the integer $\mathcal K_\tau= \lfloor \exp( \log^\tau(N) )\rfloor$ with $\tau \in (0,1)$. Moreover, let $\mathcal J=\log(N)^{r^\prime}$ with $0< r^\prime <(\tau-\tau^\prime )/r$, where $\mathcal J$ is the integer   defining $\mathcal H_J$ in \eqref{def_HJ}. Let 
	$\widehat j_0$ be the integer defined in \eqref{j0_hat45} ($0<\tau^\prime<\tau$). Then, for $N$ sufficiently large we have 
\begin{multline}
\mathbb P\left( \left| \frac{j_0}{\mathcal J} - \frac{\widehat j_0}{\mathcal J}\right| \geq \mathcal J^{-1}\right) \leq\frac{2(\mathcal J-1)}{N}+ 2 (\mathcal J-1)N\exp\left(- 2C_0^2\sqrt{\underline C_M}\mathcal K_\tau^{-2D}\mathfrak m^{1/2}\right)\\
+ 2(\mathcal J-1)\exp \left ( -  \mathfrak c \log(N)^{2\tau-2\tau^\prime} \exp(D\log(N)^\tau - 2\log(N)^{\tau - \tau^\prime}) \right),
\end{multline}
where $\mathfrak c>0$ is some constant  and  $g(\mathcal{K}_\tau)$ is defined in \eqref{g_cal}. 
\end{proposition}

\medskip

The estimator of the exponent $\alpha$ we propose is
\begin{equation}\label{def_hat_alpha}
 \hatalphamu = \left\{\widehat j_0+1/2\right\}/ \mathcal J. 
\end{equation}

\begin{corollary}\label{pour_lundi} 
	Under the conditions of Proposition \ref{pour_auj},   we have  $  \hatalphamu -\alpha = O_\PP (\log(N)^{-r^\prime})$. 
\end{corollary}
	
\medskip

Given the regularity estimator $\hatalphamu $, the optimal $L$ can also be estimated. Let $L^*(\alpha)$ and $C^*(\alpha) $ be defined as in \eqref{L_general} with $\alpha$ instead of $\alphamu$. Let $\widehat L^*=\lfloor C^*( \hatalphamu) \overline M^{1/(2\hatalphamu + D)} \rfloor $
and $C^*(\widehat \alpha) $ be their plugin estimates. 

\medskip

\begin{corollary}\label{pour_mardi} 
The conditions of Proposition \ref{pour_auj} hold, with $r< \tau -\tau^\prime<1$. Let $\mathcal J=\log(N)^{r^\prime}$ with $0< r^\prime <(\tau-\tau^\prime )/r$. If constants  $ \underline c$ and  $ \overline c$ exist such that  $ \underline c \leq \log(N)/ \log(\mathfrak m) \leq \overline c$, then 
	\begin{equation}\label{optimal_Lhat}
		\PP \left( \left|\frac{\widehat L^*}{ L^*(\alpha)}- 1\right| \geq \frac{12 (1+\underline c)/ \{D(2\alpha+D) \underline c\}}{\log (N)^{r^\prime-1}} \right) \leq 6  \mathbb P\left( \left| \frac{j_0}{\mathcal J} - \frac{\widehat j_0}{\mathcal J}\right| \geq \mathcal J^{-1}\right) .
	\end{equation}
	In particular, it holds that $\PP\{\widehat L^*= L^*(\alpha)\}\rightarrow 1$ as $N\rightarrow \infty$.
	
\end{corollary}

\medskip

Note that the condition $\log(N)/ \log(\mathfrak m) \leq \overline c$ guarantees  $\mathcal K_\tau^{2D} \mathfrak m^{-1/2}= o(\log(N))$.

\medskip

\begin{remark}
Given the estimate  $ \hatalphamu$ of the regularity $\alpha$, one would also like to estimate the associated Hölder constant $C_{\mu, \hatalphamu}$, defined as in \eqref{cond}.   For this purpose,   with $\mathfrak c>0$ some constant to be set in practice, we propose the estimate 
\begin{equation}
\widehat C_{\mu,\hatalphamu} = \min  \left( \max_{\mathbf k}\big\{\widehat b_{\mathbf k}^{1/2} v^{-1/2}_{\mathbf k}\big\},  \mathfrak c \right)
\end{equation}
where
$
v_{\mathbf k}=\{\sum_{\mathbf k^+} 1\}^{-1}\sum_{\mathbf k^+} \sum _{(i,m) ,(q,m^\prime)} \left| \mathbf T_{i,m}- \mathbf T_{q,m^\prime}\right|^{\widehat \alpha-3/(2\mathcal J)} F_{\mathbf k}(\mathbf T_{i,m})F_{\mathbf k^+}(\mathbf T_{q,m^\prime}).
$
The estimator $\widehat C_{\mu,\hatalphamu}$ is inspired by the definition of $\widehat b_{\mathbf k}$ in \eqref{def_b_kgras} and the fact that, by Proposition  \ref{pour_auj},
the event $\{ \hatalphamu  -3/(2\mathcal J)\leq \alpha\}$ has the probability tending to 1.

\end{remark}


\appendix

\section{Proofs}\label{proofs_main}

\begin{proof}[Proof of Proposition \ref{variance-V}] By definition, and since $\EE[\eta_{i,m}\eta_{i^\prime ,m^\prime }]=0$ when $i\neq i^\prime$, we have 
	\begin{multline}\label{dec_V_k}
	\EE[V_{\mathbf k}^2]= 	\EE\left[\left\{\sum_{(i,m)} \omega_{i,m}\frac{\eta_{i,m}\phi_{\mathbf k}(\Tmi)}{f_{ \mathbf T}(\Tmi)} \right\}^2 \right]\\
		 \hspace{-2cm}  =\sum_{(i,m)} \EE\left[\omega_{i,m}^2\frac{ \{\sigma^2(\Tmi)+\gamma(\Tmi,\Tmi)\}\phi_{\mathbf k}^2(\Tmi)}{f_{\mathbf T}^2(\Tmi)}  \right]
		\\\hspace{1.8cm} +\sum_{i=1}^N \sum_{1\leq m\ne m^\prime \leq M_i}\!\!\EE\!\left[ \omega_{i,m}\omega_{i,m^\prime}\frac{\gamma(\Tmi, \mathbf T_{i,m^\prime})\phi_{\mathbf k}(\Tmi)\phi_{\mathbf k}(\mathbf T_{i,m^\prime})}{f_{\mathbf T}(\Tmi)f_{ \mathbf T}(\mathbf T_{i,m^\prime})}\right]\\
 =: \sum_{(i,m)} \EE\left[\omega_{i,m}^2\frac{ \{\sigma^2(\Tmi)+\gamma(\Tmi,\Tmi)\}\phi_{\mathbf k}^2(\Tmi)}{f_{ \mathbf T}^2(\Tmi)}  \right]
+\mathcal R_{\mathbf k}.
	\end{multline}
Next, for any $\mathbf s,\mathbf t\in \mathcal T$,  by assumptions (H\ref{assuH12}), (D\ref{assuD4}) and (D\ref{assuD5}), we have
\begin{multline}
\left| \frac{ \{\sigma^2(\mathbf t)+\gamma(\mathbf t,\mathbf t)\}\phi_{\mathbf k}^2(\mathbf t)}{f_{\mathbf T}^2(\mathbf t)}- \frac{ \{\sigma^2(\mathbf s)+\gamma(\mathbf s,\mathbf s)\}\phi_{\mathbf k}^2(\mathbf s)}{f_{\mathbf T}^2(\mathbf s)}\right|\\ 
 \leq \frac{4C_1^2}{C_0^4} \|\sigma^2 + \gamma\|_\infty \|\mathbf t - \mathbf s \|^{\alpha_{f}}+ 
 \frac{2}{C_0^2}\{\|\mathbf  t-\mathbf s\|^{\alpha_{\sigma}}+ \|\mathbf t-\mathbf s\|^{H_\gamma}\} 
 + \frac{ 2^{D+1}\pi   \| \sigma^2 + \gamma\|_\infty}{C_0^2} |\mathbf k|_1 \|\mathbf t-\mathbf s \|.
\end{multline}
Here,  $(\sigma^2 + \gamma)(\mathbf t)=\sigma^2(\mathbf t)+\gamma (\mathbf t,\mathbf t)$.
Using Proposition \ref{lem_Vor_Index}, a constant $\varrho(D)$ exists such that   
\begin{multline}
\sum_{(i,m)} \EE\left[\omega_{i,m}^2\frac{\{ \sigma^2(\Tmi)+\gamma(\Tmi,\Tmi)\}\phi_{\mathbf k}^2(\Tmi)}{f_T^2(\Tmi)}  \right] \\ = \frac{1}{\overline M^2} \sum_{(i,m)} \EE\left[\{1 + (\widehat c_{i,m} - \widehat d_{i,m})^2 \}\frac{ \sigma^2(\Tmi)\phi_{\mathbf k}^2(\Tmi)}{f_T^2(\Tmi)}  \right] \\
= \frac{\varrho(D)}{\overline M} \EE\left[\frac{ \{\sigma^2(\mathbf T)+\gamma(\mathbf T,\mathbf T)\}\phi_{\mathbf k}^2(\mathbf T)}{f_{\mathbf T}^2(\mathbf T)}  \right]\times\left \{ 1 + O\left(\overline M^{\; -\underline \alpha} \right) + |\mathbf k|_1 O\left( \overline M^{\;-1}\right)  \right \},
\end{multline}  
with $\underline \alpha=\min\{\alpha_f,\alpha_\sigma, H_\gamma\}\in(0,1]$. We show in  the Supplement that $\varrho (D)=5/2$ if $D=1$.  
For $D>1$, \cite{kabatiansky1978bounds} have shown that $\widehat d_{i,m} \lesssim  \mathfrak K^{D}$, see also \citet[Lemma 1.3]{Henze87}, for some constant $\mathfrak R$.  On the other hand,  \citet[Theorem 2.1]{devroye2017} have shown that $ \overline M ^2\mathbb E[ V_{i,m}^2\mid \mathbf T_{i,m}] \rightarrow \alpha(D)\in[1,2]$, where $ V_{i,m}$ is the volume of the Voronoi cell of $\mathbf T_{i,m}$ computed with respect to the distribution of $\mathbf T$. These facts imply $\varrho (D)\lesssim \mathfrak K^{D}$.

The following lemma, proved in  the Supplement, is used for bounding  $\mathcal R_{\mathbf k}$ in \eqref{dec_V_k}.

\medskip

\begin{lemma}\label{qar1}
Let $\psi  : \mathcal T^2 \rightarrow \mathbb R$ be a Hölder continuous function. Then, under the Assumptions (H\ref{assuH2}) and (H\ref{assuH12}), we have for $m\ne m^\prime$,
\begin{equation}
\EE\left[ \omega_{i,m}\omega_{i,m^\prime} \psi(\Tmi, \mathbf T_{i,m^\prime})\right] = \frac{1}{\overline M(\overline M - 1)}\int_{\mathcal T^2} \psi(\mathbf s, \mathbf t) f_{\mathbf T}(\mathbf s)f_{\mathbf T}(\mathbf t) {\rm d}\mathbf s {\rm d}\mathbf t \times\{ 1 + o(1)\}.
\end{equation}
\end{lemma}

Applying  Lemma \ref{qar1} to the function 
$$
(\mathbf s, \mathbf t) \mapsto \gamma(\mathbf s, \mathbf t)\phi_{\mathbf k}(\mathbf s)\phi_{\mathbf k}(\mathbf t)f^{-1}_{\mathbf T}(\mathbf s)f^{-1}_{ \mathbf T}(\mathbf t),
$$
which is $\underline \alpha$- Hölder continuous by (D\ref{assuD4}) and (D\ref{assuD5}),
 $\underline \alpha=\min\{\alpha_f,\alpha_\sigma, H_\gamma\}$, we get  for any $\mathbf k$,
$$
\mathcal R_{\mathbf k} = \frac{ \sum_{i=1}^N M_i(M_i-1)}{\overline M(\overline M - 1)}\int_{\mathcal T^2} \gamma(\mathbf s, \mathbf t)\phi_{\mathbf k}(\mathbf s)\phi_{\mathbf k}(\mathbf t) {\rm d}\mathbf s{\rm d}\mathbf t \times\{ 1 + o(1)\},
$$
with the $(1)$ not depending on $\mathbf k$.
Now the proof of Proposition \ref{variance-V} is complete. \end{proof}

\medskip

\begin{proof}[Proof of Proposition \ref{mean:L2-norm}]
For \eqref{bterms}, by Proposition \ref{variance-B} and summing over $\mathbf k$, we get 
\begin{multline}
	\frac{1}{L} \sum_{j=0}^{L-1} \sum_{| \mathbf k |_1\leq 2(L+j)}\EE(B_{ \mathbf k} - a_{\mathbf k})^2 \lesssim \frac{1}{L} \sum_{j=0}^{L-1} \sum_{|\mathbf k|_1 \leq 2(L+j)}
	\left\{ \overline M ^{\; -(1+2\alphamu/D)} +|\mathbf k|^2  \overline M ^{\;- (1 + 2/D)}\right\}\\
\lesssim L^{D\; } \overline M ^{\; -(1+2\alphamu/D)} +  L^{ D+2\; } \overline M^{\; -(1 + 2/D)} .
\end{multline}
The constants in the $\lesssim$ are determined by the constants in Proposition \ref{variance-B}.
For justifying \eqref{Vterms}, first,We show in the Supplement that, for any $\mathbf t \in \mathcal T$
$$
\frac{1}{L} \sum_{j=0}^{L-1} \sum_{|\mathbf k|_1\leq 2(L+j)} \phi_{\mathbf k}^2(\mathbf t)=\frac{2^{2D+1} -2^D}{ (D+1)!} L^D + O(L^{D-1}).
$$
 Therefore, we get
\begin{multline}\label{pige_enfin}
	\!\!\! \frac{1}{L} \sum_{j=0}^{L-1} \sum_{|\mathbf k|_1\leq 2(L+j)}\!\!\!	\EE\!\left[\!\frac{ \{\sigma^2(\mathbf T)\!+\!\gamma(\mathbf T,\mathbf T)\}\phi_{\mathbf k}^2(\mathbf T)}{f_{\mathbf T}^2(\mathbf T)}  \!\right]\! \!= \frac{2^{2D+1}\!\! -\!2^D}{ (D+1)!} L^D\!\! \int_{\mathcal T}\!\frac{\gamma(\mathbf s, \mathbf s)\!+\! \sigma^{2} (\mathbf s)}{f_{\mathbf T}(\mathbf s)}{\rm d}\mathbf s + O(L^{D-1}).
\end{multline}
Finally, using Assumption (D\ref{assuD5}) and Proposition \ref{reste-control} on the approximation error of the de La Vallée Poussin sum, applied for each $\mathbf t\in\mathcal T$ to the map $\mathbf s\mapsto \gamma(\mathbf s,\mathbf t)$, we get 
\begin{multline}
\!\!\!\! \frac{1}{L} \sum_{j=0}^{L-1} \sum_{|\mathbf k|_1\leq 2(L+j)} \iint_{\mathcal T^2} \!\gamma(\mathbf s,\mathbf t)\phi_{\mathbf k}(\mathbf s) \phi_{\mathbf k}(\mathbf t){\rm d}\mathbf s{\rm d}\mathbf t = \int_{\mathcal T}\frac{1}{L} \sum_{j=0}^{L-1} \sum_{|\mathbf k|_1\leq 2(L+j)} \int_{\mathcal T} \!\gamma(\mathbf s,\mathbf t)\phi_{\mathbf k}(\mathbf s){\rm d}\mathbf s \phi_{\mathbf k}(\mathbf t) {\rm d} \mathbf t \\
  = \int_{\mathcal T} \gamma(\mathbf t,\mathbf t){\rm d}\mathbf t\times \{1 +   O\left(L^{-H_\gamma}\right)   \}.
\end{multline}
\end{proof}

\begin{proof}[Proof of Theorem \ref{normal_approximation}]
	Aiming to apply  \citet[Th. 2.1]{MVNAReinert2009}, the idea is to construct a random vector $ W ^\prime(d)$ such that the pair $ ( W(d),  W^\prime(d))$ is exchangeable and 
\begin{equation}\label{exch_cond}
\EE\left[  W^\prime(d) - W (d)\mid  W(d)\right]= -\Lambda  W (d) + R,
\end{equation}
for an invertible $(\sum_{|\mathbf k|_1\leq d} 1)\times(\sum_{|\mathbf k|_1\leq d}1) $ matrix $\Lambda$ and a remainder vector term $R$.   Recall that $W (d)\in\mathbb R^{\overline d}$ was defined in \eqref{eq:def_W_Stein} with the components $V_{\mathbf k}$, $|\mathbf k|_1\leq d$.   Here, the construction of $W^\prime(d)$ is done as follows~: consider \vspace{-.1cm} 
\begin{itemize}
\item $\{X_1^\prime, \dots, X_N^\prime \}$ an independent copy of  the realizations $ \{X_1, \dots, X_N\}$; \vspace{-.1cm} 
\item \hspace{-.2cm} $\{\varepsilon^\prime_{i,m}\!:  1\leq i\leq N, 1\leq m\leq M_i\}$   independent copy of    $\{\varepsilon_{i,m}\!:  1\leq i\leq N, 1\leq m\leq M_i\}$;\vspace{-.2cm} 
\item \hspace{-.2cm}  $\mathbf S\in \mathbb R^{\overline d}$ a uniform random vector with the support $\{\mathbf k\in  \mathbb N^D: |\mathbf k|_1\leq  d\}$;  \vspace{-.1cm} 
\item \hspace{-.2cm} $(I, J)$ a uniform random pair on $\{(i,m) :  1\leq i\leq N, 1\leq m\leq M_i\}$; \vspace{-.1cm} 

\item \hspace{-.2cm}  $\{ \mathbf u_{\mathbf k}\}_{|\mathbf k|_1 \leq d} $ the canonical basis of $\mathbb R^{\overline d}$.
\end{itemize}
Next, for $\eta_{i,m}= \varepsilon_{i,m} + \{X_i - \mu\}(\mathbf T_{i,m})$, we set  $\eta_{i,m}^\prime= \varepsilon_{i,m}^\prime + \{X_i^\prime - \mu\}(\mathbf T_{i,m})$ and
\begin{equation}\label{def_W_Stein_bis}
 W^\prime(d)= W (d)+
 \omega_{I,J} \frac{\phi_{\mathbf S}(\mathbf T_{I,J})}{f_{\mathbf T}(\mathbf T_{I,J})} \{\eta^\prime_{I,J} - \eta_{I,J} \}\mathbf u_{\mathbf S}. 
\end{equation}
Then $(W(d), W^\prime(d))$ is exchangeable and the condition \eqref{exch_cond} holds with $R=0$ and $\Lambda$ a diagonal matrix, \emph{i.e.}, $\Lambda = \{\overline d \; \overline M\}^{-1} I_{\overline d}$.   For readability, hereafter the components of $W(d)$ are renamed $W_{\mathbf k}$ and those of $W^\prime (d)$ are denoted $W^\prime_{\mathbf k}$.   By Proposition \ref{variance-V}, for $ \mathbf k \in \mathbb N^D$, 
\begin{multline}
	\operatorname{Var}(W_{\mathbf k}) = \mathbb E[ W_{\mathbf k}^2]= \frac{\varrho(D)}{\overline M} \int_{\mathcal T} \frac{ \{\sigma^2(\mathbf t)+\gamma(\mathbf t,\mathbf t)\}\phi_{\mathbf k}^2(\mathbf t) }{f_{\mathbf T}(\mathbf t)}{d}\mathbf t\times\left \{ 1  + o(1) \right \}
\\
+\frac{ \sum_{i=1}^N M_i(M_i-1)}{\overline M(\overline M - 1)}\iint_{\mathcal T^2} \gamma(\mathbf t, \mathbf s)\phi_{\mathbf k}(\mathbf t)\phi_{\mathbf k}(\mathbf s) {\rm d}\mathbf t{\rm d}\mathbf s \times\{ 1 + o(1)\}),
\end{multline}
with the $o(\cdot)$  not depending on $\mathbf k$. 	
Let $\EE^W[\cdots]=\EE [\cdots \mid W]$. By the definition \eqref{def_W_Stein_bis}, we have 
\begin{multline}
\EE^W\left[ (W_{\mathbf k}^\prime - W_{\mathbf k})^2\right]= \EE^W\left[\mathbbm1_{\{\mathbf  S=\mathbf k\}} (W_{\mathbf k}^\prime - W_{\mathbf k})^2\right]\\
=\frac{1}{\overline d} \EE^W\left[  \omega_{I,J}^2 \frac{\phi_{\mathbf k}^2(\mathbf T_{I,J})}{f_{\mathbf T}^2(\mathbf T_{I,J})} \{\eta^\prime_{I,J} - \eta_{I,J} \}^2\right]\\
=\frac{1}{\overline d\;\overline M} \sum_{(i,m)} \omega_{i,m}^2 \frac{\phi_{ \mathbf k}^2(\mathbf T_{i,m})}{f_{\mathbf T}^2(\mathbf T_{i,m})}\left\{\sigma^2(\mathbf T_{i,m}) + \gamma(\mathbf T_{i,m}, \mathbf T_{i,m})  + \eta_{i,m}^2\right\}.
\end{multline}
Therefore, we deduce
\begin{multline}
\EE\left[ \EE^W\left[ (W_{\mathbf k}^\prime - W_{\mathbf k})^2\right]\right] =\frac{2}{\overline d\;\overline M} \sum_{(i,m)} \EE\left[\omega_{i,m}^2 \frac{\phi_{ \mathbf k}^2(\mathbf T_{i,m})}{f_{\mathbf T}^2(\mathbf T_{i,m})}\left\{\sigma^2(\mathbf T_{i,m}) + \gamma(\mathbf T_{i,m},\mathbf T_{i,m}) \right\}\right]\\
= \frac{2\varrho(D)}{\overline d\;\overline M^{\;2}} 
\int_{\mathcal T}
\frac{ \{\sigma^{2}(\mathbf s)+\gamma(\mathbf s,\mathbf s)\}\phi_{ \mathbf k}^2(\mathbf s)}{f_{ \mathbf T}(\mathbf s)}{\rm d}\mathbf s\times 
\left\{ 1 + o(1)\right\}.
\end{multline}
Next,
\begin{multline}
\EE\left[ \EE^W\left[ (W_{\mathbf k}^\prime - W_{\mathbf k})^2\right]^2\right]\\
= \frac{1}{\overline d ^2\overline M^{\; 2}} \EE\bigg[ \bigg\{\sum_{(i,m)} \omega_{i,m}^2 \frac{\phi_{ \mathbf k}^2(\mathbf T_{i,m})}{f_{\mathbf T}^2( \mathbf T_{i,m})}\bigg\{\sigma^2(\mathbf T_{i,m}) + \gamma(\mathbf T_{i,m}, \mathbf T_{i,m})  + \eta_{i,m}^2\bigg\}\bigg\}^2 \bigg] \\
=\frac{3}{\overline d^2\overline M^{\; 2}} \sum_{(i,m)}\sum_{(j,m^\prime)}\EE\bigg[ \omega_{i,m}^2\omega_{j,m^\prime}^2 \frac{\phi_{ \mathbf k}^2(\mathbf T_{i,m})}{f_{\mathbf T}^2( \mathbf T_{i,m})} \frac{\phi_{ \mathbf k}^2(\mathbf T_{j,m^\prime})}{f_{\mathbf T}^2( \mathbf T_{j,m^\prime})}\left\{\sigma^2(\mathbf T_{i,m}) + \gamma(\mathbf T_{i,m},\mathbf T_{i,m}) \right\} \\
 \times\bigg\{\sigma^2(\mathbf T_{j,m^\prime}) + \gamma(\mathbf T_{j,m^\prime}, \mathbf T_{j,m^\prime}) \bigg\} \bigg] + \mathcal R_1,
\end{multline}
where
\begin{equation}
 \mathcal R_1 = \frac{1}{\overline d^2\overline M^{\; 2}}\sum_{(i,m)}\sum_{(j,m^\prime)}\EE\left[ \omega_{i,m}^2\omega_{j,m^\prime}^2 \frac{\phi_{ \mathbf k}^2(\mathbf T_{i,m})}{f_{\mathbf T}^2( \mathbf T_{i,m})} \frac{\phi_{ \mathbf k}^2(\mathbf T_{j,m^\prime})}{f_{\mathbf T}^2( \mathbf T_{j,m^\prime})}\eta_{i,m}^2\eta_{j,m^\prime}^2\right].
\end{equation}
For any   index $(j,m^{\prime})$, the number of observation points needed to construct $ \widehat d_{j,m^\prime}$ and $\widehat c_{j, m^\prime}$ is bounded by a constant, say, $C_D$ depending only on   $D$.   We thus obtain by independence  
\begin{equation}
\EE\left[ \EE^W\left[ (W_{\mathbf k}^\prime - W_{\mathbf k})^2\right]^2\right]  =
  \frac{ 3\varrho^2(D) }{ \overline d^2 \overline M ^{\; 4}} 
\times \left(\!
\int_{\mathcal T}
 \frac{\phi_{\mathbf k} ^2(\mathbf t)}{f_{\mathbf T}(\mathbf t)}\{\sigma^2(\mathbf t) + \gamma(\mathbf t,\mathbf t)\}{\rm d}\mathbf t \!\right)^2 \times \{1+o(1)\} + \mathcal R_1.
\end{equation}
 For the remainder term $\mathcal R_1$, we have 
\begin{multline}
\overline d^2\overline M^{\; 2} \mathcal R_1 = \sum_{(i,m)}\sum_{(j,m^\prime)}\EE\left[ \omega_{i,m}^2\omega_{j,m^\prime}^2 \frac{\phi_{ \mathbf k}^2(\mathbf T_{i,m})}{f_{\mathbf T}^2( \mathbf T_{i,m})} \frac{\phi_{ \mathbf k}^2(\mathbf T_{j,m^\prime})}{f_{\mathbf T}^2( \mathbf T_{j,m^\prime})}\eta_{i,m}^2\eta_{j,m^\prime}^2\right]\\
=\sum_{(i,m)}\sum_{(j,m^\prime): j\ne i}\EE\left[ \omega_{i,m}^2\omega_{j,m^\prime}^2 \frac{\phi_{ \mathbf k}^2(\mathbf T_{i,m})}{f_{\mathbf T}^2( \mathbf T_{i,m})} \frac{\phi_{ \mathbf k}^2(\mathbf T_{j,m^\prime})}{f_{\mathbf T}^2( \mathbf T_{j,m^\prime})}\left\{\sigma^2(\mathbf T_{i,m}) + \gamma(\mathbf T_{i,m}, \mathbf T_{i,m}) \right\}\right.\\
\left. \times\left\{\sigma^2(\mathbf T_{j,m^\prime}) + \gamma(\mathbf T_{j,m^\prime},\mathbf T_{j,m^\prime}) \right\} \right]\\
+ \sum_{i=1}^N \sum_{ m, m^\prime=1}^{M_i} \EE\left[ \omega_{i,m}^2\omega_{i,m^\prime}^2 \frac{\phi_{ \mathbf k}^2(\mathbf T_{i,m^\prime})}{f_{\mathbf T}^2( \mathbf T_{i,m^\prime})} \frac{\phi_{ \mathbf k}^2(\mathbf T_{i,m^\prime})}{f_{\mathbf T}^2( \mathbf T_{i,m^\prime})}\left\{\sigma^2(\mathbf T_{i,m})\sigma^2(\mathbf T_{i,m^\prime} )+ \sigma^2(\mathbf T_{i,m})\gamma(\mathbf T_{i,m^\prime}, \mathbf T_{i,m^\prime}) \right. \right.\\
\left.\left.+ \sigma^2(\mathbf T_{i,m^\prime})\gamma(\mathbf T_{i,m},\mathbf  T_{i,m}) + \{ X_i-\mu\}^2(\mathbf T_{i,m})\{X_i-\mu\}^2(\mathbf T_{i,m^\prime}) \right\} \right].
\end{multline}
Since the constant $C_D$ is independent of the sample size, we have~:
\begin{multline}
\overline d^2\overline M^{\; 2} \mathcal R_1=\varrho^2(D) \times \frac{\sum_{i=1}^N M_i(\overline M \!-M_i \!- 2C_D)}{ \overline M^{\; 4}} \left(
\int_{ \mathcal T}
\frac{\phi_{\mathbf k} ^2(\mathbf t)}{f_{ \mathbf T}(\mathbf t)}\{\sigma^2(\mathbf t) \!+\! \gamma(\mathbf t, \mathbf t)\}{\rm d}\mathbf t \right)^2 \! \times \{1+o(1)\} \\
+\varrho^2(D)  \times \frac{\sum_{i=1}^N M_i(M_i-2C_D)}{\overline M^{\; 4}}\left\{
2
\int_{\mathcal T}
\frac{\phi_{\mathbf k} ^2(t)}{f_{ \mathbf T}(\mathbf t)}\sigma^2(\mathbf t){\rm d} \mathbf t 
\times \int_{\mathbf t\in\mathcal T}
\frac{\phi_{ \mathbf k} ^2(\mathbf t)}{f_{ \mathbf T}(\mathbf t)}\gamma(\mathbf t,\mathbf t){\rm d}\mathbf t \right.\\ +  \left( 
\int_{\mathcal T}
\frac{\phi_{\mathbf k} ^2(\mathbf t)}{f_{\mathbf T}(\mathbf t)}\sigma^2(\mathbf t){\rm d}\mathbf t \right)^2
+ \left. 
\iint_{\mathcal T^2}
\frac{\phi_{ \mathbf k} ^2(\mathbf t)}{f_{ \mathbf T}(\mathbf t)} \frac{\phi_{ \mathbf k} ^2(\mathbf s)}{f_{\mathbf T}(\mathbf s)}m_4(\mathbf t,\mathbf s){\rm d}\mathbf t {\rm d}\mathbf s \right\} \times \{ 1+ o(1)\},
\end{multline}
where 
$
m_4(\mathbf t,\mathbf s)= \EE\left[\{X-\mu\}^2(\mathbf t)\{X-\mu\}^2(\mathbf s)  \right].
$
Thus, concerning the quantity $A$ in \citep[Theorem 1]{MVNAReinert2009}, we have
\begin{multline}
\operatorname{Var}\left(  \EE^W\left[ (W_{\mathbf k}^\prime - W_{\mathbf k})^2\right]\right) = \EE\left[ \EE^W\left[ (W_{\mathbf k}^\prime - W_{\mathbf k})^2\right]^2 \right]- \EE\left[ (W_{\mathbf k}^\prime - W_{\mathbf k})^2\right]^2\\
\leq \varrho^2(D) \bigg\{ \frac{2C_D  }{ \overline d^2 \overline M^{\;5}}\mathfrak a_1 +
 \frac{1}{\overline d^2 \overline M^{\; 6}} \sum_{i =1}^N  M_i(M_i+2C_D)( \mathfrak a_1+\mathfrak a_2)\bigg\} \times \{1+o(1)\},
\end{multline}
where 
$$
\mathfrak a_1=\left(
\int_{\mathcal T}
\frac{2^{D}}{f_{\mathbf T}(\mathbf t)}\{\sigma^2(\mathbf t) + \gamma(\mathbf t,\mathbf t)\}{\rm d}\mathbf t \right)^2,
$$
and
\begin{multline}
\mathfrak a_2= \left( 
\int_{\mathcal T}
\frac{2^{D}}{f_{ \mathbf T}(\mathbf t)}\sigma^2(\mathbf t){\rm d}\mathbf t \right)^2
\\ +2
\int_{\mathcal T}
\frac{2^{D}}{f_{\mathbf T}(\mathbf t)}\sigma^2(\mathbf t){\rm d}\mathbf t 
\int_{\mathcal T}
\frac{2^{D}}{f_{\mathbf T}(\mathbf t)}\gamma(\mathbf t,\mathbf t){\rm d}\mathbf t +  
\iint_{\mathcal T^2}
\frac{2^{2D}}{f_{\mathbf T}(\mathbf t)f_{ \mathbf T}(\mathbf s)}m_4(\mathbf t,\mathbf s){\rm d}\mathbf t {\rm d}\mathbf s. 
\end{multline}
Concerning the quantity $B$ in \citep[Th. 2.1]{MVNAReinert2009}, we have
\begin{multline}
\EE\left[|W^\prime_{\mathbf k}-W_{\mathbf k} |^3 \right] = \frac{ 1}{\overline  d} \EE\left[\left| \omega_{I,J}^3 \frac{\phi_{ \mathbf k}^3(\mathbf T_{I,J})}{f_{ \mathbf T} ^3(\mathbf T_{I,J})} \{\eta^\prime_{I,J} - \eta_{I,J} \}^3 \right| \right]\\ 
= \frac{ 1}{ \overline d \; \overline M} \sum_{(i,m)}\EE\left[\left| \omega_{i,m}^3 \frac{\phi_{\mathbf k}^3(\mathbf T_{i,m})}{f_{\mathbf T}^3(\mathbf T_{i,m})} \{\eta^\prime_{i,m} - \eta_{i,m} \}^3 \right| \right]
\leq\frac{ \mathfrak c_1}{ \overline d\;\overline M^{\;3}},
\end{multline}
for some constant  $\mathfrak c_1$.

By the definition \eqref{def_W_Stein_bis},  $(W_{\mathbf k}^\prime - W_{\mathbf k}) (W_{\mathbf j }^\prime - W_{\mathbf j})=0$ if $\mathbf k \neq \mathbf j$. This means that all the terms in the sums defining the quantities $A$ and $B$ in \citep[Th. 2.1]{MVNAReinert2009} are equal to zero, except those with indices $i=j$ and $i=j=k$, respectively. These diagonal terms have been bounded above. Moreover, the quantity $C$  in that Theorem 2.1 is equal to zero.
Gathering facts, for any three times differentiable function $h$, we deduce
\begin{multline}
\left|\EE[ h( W) - \EE h(\Sigma ^{1/2} Z)]\right| \leq   \frac{|h|_3\overline d\mathfrak c_1}{12 \overline M^{\;2}}\\
+\frac{\overline d \varrho(D)|h|_2}{4}\bigg\{ \frac{2C_D  }{ \overline M^{\;3}}\mathfrak a_1 +
 \frac{1}{\overline M^{\; 4}} \sum_{i =1}^N  M_i(M_i+2C_D)( \mathfrak a_1+\mathfrak a_2)\bigg\} ^{1/2},
\end{multline}
and this proves our Theorem \ref{normal_approximation}.
\end{proof}

\medskip

Before proving the Proposition \ref{pas_encore}, let us first decompose $\EE\left[ \widehat b_{\mathbf k} \right]$, for $|\mathbf k|_1\leq D(\mathcal K-1)-1$ and bound the terms of the decomposition. More precisely, let $\mathbf k^+$ such that $|\mathbf k^+|_1 = | \mathbf k|_1+1$, we write
\begin{multline}\label{def_les_frakC}
\EE \Bigg [\Bigg\{ \sum _{(i, m)}  Y_{ i,m} \{F_{\mathbf k}(\mathbf T_{i,m}) -  F_{\mathbf k^+}(\mathbf T_{i,m})\} \Bigg\}^{\!2\;} \Bigg]  
=\EE\Bigg[  \Bigg( \sum _{(i,m)} Y_{i,m}F_{\mathbf k}(\mathbf T_{i,m})\Bigg)^2\Bigg] \\ + \EE\Bigg[ \Bigg(\sum _{(i,m)} ( Y_{i,m}) F_{\mathbf k^+}(\mathbf T_{i,m})\Bigg)^2\Bigg]
 - 2 \EE\Bigg[\sum _{(i,m) \ne (q,m^\prime)}   Y_{i,m}  Y_{q,m^\prime}F_{\mathbf k}(\mathbf T_{i,m}) F_{\mathbf k^+}(\mathbf T_{q,m^\prime})  \Bigg] 
\\ =:\mathfrak C_{1,\mathbf k}+\mathfrak C_{2,\mathbf k}-2\mathfrak C_{3,\mathbf k}.
\end{multline}
Bounds for the terms $\mathfrak C_{1,\mathbf k}$, $\mathfrak C_{2,\mathbf k}$, $\mathfrak C_{3,\mathbf k}$ are given below and proved in the Supplement. 

\medskip

\begin{lemma}[Bounds for $\mathfrak C_{1,\mathbf k}$, $\mathfrak C_{2,\mathbf k}$ in \eqref{def_les_frakC}]\label{lem1_pas_encore}
	Under (H\ref{assuH2}),  (H\ref{assuH12}) and  (D\ref{assuD1}) to (D\ref{assuD5}), we have
\begin{equation}
\mathfrak C_{1,\mathbf k}=\EE\Bigg[  \Bigg( \sum _{(i,m)}  Y_{i,m}F_{\mathbf k}(\mathbf T_{i,m})\Bigg)^{\!2\;}\Bigg] 
=  \frac{1}{p_{\mathbf k}^2 } \EE\left[\mu(\mathbf T_1)\mu(\mathbf T_2)\mathbbm1_{\{ \mathbf T_1,\mathbf T_2\in I_{\mathbf k}\}} \right]+ \mathcal G_{\mathbf k},
\end{equation}
and
\begin{equation}
 \left|\mathcal G_{\mathbf k}\right|  \lesssim  \frac{\mathcal K^D}{C_0 \overline M }  \times\left\{ \mu^2(\mathbf  t_{\mathbf k}) +\sigma^2(\mathbf  t_{\mathbf k}) + \gamma(\mathbf  t_{\mathbf k}, \mathbf  t_{\mathbf k})\right\} + \frac{\sum_{i=1}^N M_i(M_i-1)}{\overline M (\overline M -1)} \gamma(\mathbf   t_{\mathbf k},\mathbf  t_{\mathbf k}).
\end{equation}
A similar bound is valid for $\mathfrak C_{2,\mathbf k}$ by  replacing $p_{\mathbf k}$, $I_{\mathbf k}$, $ \mathbf t_{\mathbf k}$ by $p_{\mathbf k^+}$, $I_{\mathbf k^+}$, $ \mathbf t_{\mathbf k^+}$, respectively. 
\end{lemma}


\medskip

\begin{lemma}[Representation of $\mathfrak C_{3,\mathbf k}$ in \eqref{def_les_frakC}]\label{lem2_pas_encore}
We have the following identity
\begin{multline}\label{eq:crois_lem}
\mathfrak C_{3,\mathbf k}^{(j)}
	= \frac{1}{p_{\mathbf k}p_{\mathbf k^+}}\EE\left[\mu(\mathbf T_1)\mu(\mathbf T_2)\mathbbm1_{\{ (\mathbf T_1,\mathbf T_2)\in I_{\mathbf k} \times I_{\mathbf k^+}\}}  \right]\times \{1- \rho_3\}\\
	+\frac{\sum_{i=1}^N M_i(M_i-1)}{\overline M (\overline M-1)}\EE\left[\gamma( \mathbf T_1,\mathbf T_2) \mathbbm1_{\{ (\mathbf T_1,\mathbf T_2)\in I_{\mathbf k} \times I_{\mathbf k^+}\}}\right]\times\{1- \rho_3\},
\end{multline}
with 
$
\rho_3= (1-p_{\mathbf k})^{\overline M}+ (1-p_{\mathbf k^+})^{\overline M} - (1-p_{\mathbf k}-p_{\mathbf k^+})^{\overline M}.
$
\end{lemma}

\medskip

\begin{proof}[Proof of Proposition \ref{pas_encore}]

Gathering Lemma \ref{lem1_pas_encore}, Lemma \ref{lem2_pas_encore} and \eqref{def_les_frakC},  we have
\begin{multline}
\EE \Bigg [\Bigg\{ \sum _{(i, m)}  Y_{ i,m} \{F_{\mathbf k}(\mathbf T_{i,m}) -  F_{\mathbf k^+}(\mathbf T_{i,m})\} \Bigg\}^2 \Bigg]  
= \frac{1}{p_{\mathbf k}^2 } \EE\left[ \mu(\mathbf T_1)\mu(\mathbf T_2)\mathbbm1_{\{ \mathbf T_1, \mathbf T_2\in I_{\mathbf k}\}} \right] \\
+  \frac{1}{p_{\mathbf k^+}^2 } \EE\left[\mu(\mathbf T_1)\mu(\mathbf T_2)\mathbbm1_{\{ \mathbf T_1,\mathbf T_2\in I_{\mathbf k^+}\}} \right]
- 2 \frac{1}{p_{\mathbf k}p_{\mathbf k^+}}\EE\left[\mu(\mathbf T_1)\mu(\mathbf T_2)\mathbbm1_{\{ (\mathbf T_1,\mathbf T_2)\in I_{\mathbf k} \times I_{\mathbf k^+}\}}  \right] 
+ \mathcal G\\
=   \frac{1}{p_{\mathbf k}^2p_{\mathbf k^+}^2}\EE\left[\left\{   \mu(\mathbf T_1) - \mu(\mathbf T_2)\right\} \left\{\mu(\mathbf T_1^\prime) -\mu(\mathbf T_2^\prime)\right\}  \mathbbm1_{\{\mathbf T_1,\mathbf T_1^\prime\in I_{\mathbf k} \}}\mathbbm1_{\{ \mathbf T_2,\mathbf T_2^\prime\in  I_{\mathbf k^+}\}}\right] +\mathcal G\\
=  \frac{1}{p_{\mathbf k}^2p_{\mathbf k^+}^2}\EE\left[\left\{   \mu(\mathbf T_1) - \mu(\mathbf T_2)\right\}  \mathbbm1_{\{ (\mathbf T_1,\mathbf T_1)^\prime\in I_{\mathbf k}\times I_{\mathbf k^+}  \}}\right] ^2+\mathcal G ,
\end{multline}
where  $\mathbf T_1, \mathbf T_2, \mathbf T_1^\prime$, $\mathbf T_2^\prime$ are random copies of $\mathbf T$ and $\mathcal G$ is a reminder that can be decomposed
\begin{multline}
\mathcal G=\mathcal G_{\mathbf k} +\mathcal G_{\mathbf k^+} -2\frac{\sum_{i=1}^N M_i(M_i-1)}{p_{\mathbf k} p_{\mathbf k^+}\overline M (\overline M-1)}\EE\left[\gamma( \mathbf T_1,\mathbf T_2) \mathbbm1_{\{ (\mathbf T_1,\mathbf T_2)\in I_{\mathbf k} \times I_{\mathbf k^+}\}}\right]\times\{1- \rho_3\}\\
+ \frac{2}{p_{\mathbf k}p_{\mathbf k^+}}\EE\left[ \mu(\mathbf T_1)\mu(\mathbf T_2)\mathbbm1_{\{ (\mathbf T_1,\mathbf T_2)\in I_{\mathbf k} \times I_{\mathbf k^+}\}}  \right]\times  \rho_3,
\end{multline}
with $\mathcal G_{\mathbf k}$ as in Lemma \ref{lem1_pas_encore} and 
$
\rho_3= (1-p_{\mathbf k})^{\overline M}+ (1-p_{\mathbf k^+})^{\overline M} - (1-p_{\mathbf k}-p_{\mathbf k^+})^{\overline M}.
$
With similar calculations as used to bound  $\mathcal G_{\mathbf k}$, we get    the following bound  for the remainder $\mathcal G$~: 
\begin{multline}
 \left| \mathcal G\right|  \lesssim  \frac{\mathcal K^D}{C_0 \overline M } \times\left\{ \mu^2 ( \mathbf  t_{\mathbf k}) + \mu^2(\mathbf  t_{\mathbf k^+})+\sigma^2( \mathbf  t_{\mathbf k}) +\sigma^2(\mathbf  t_{\mathbf k^+})+ \gamma(\mathbf  t_{\mathbf k}, \mathbf  t_{\mathbf k})  + \gamma(\mathbf  t_{\mathbf k^+}, \mathbf  t_{\mathbf k^+})\right\} \\+ \frac{\sum_{i=1}^N M_i(M_i-1)}{\overline M (\overline M -1)}\{ \gamma(\mathbf  t_{\mathbf k}, \mathbf  t_{\mathbf k}) + \gamma(\mathbf  t_{\mathbf k^+}, \mathbf  t_{\mathbf k^+})-2\gamma( \mathbf  t_{\mathbf k}, \mathbf  t_{\mathbf k^+}) \}.  
\end{multline}
\end{proof}

\begin{proof}[Proof of Proposition \ref{pour_demain}]
Recall that $g(\mathcal K)$ was defined in \eqref{g_cal}, and set 
\begin{equation}
\widehat{ g}(\mathcal K)=\frac{1}{\mathcal K^D-1} \sum_{\mathbf k \in \mathcal X(\mathcal K)}\widehat b_{\mathbf k}.
\end{equation}
By the definitions of $\widehat H_j$ and $\widetilde H_j$ in \eqref{estimator_Hj} and \eqref{proxy_Hj}, for any $\eta >0$, we have 
\begin{multline}
\mathbb P \left(|\widehat H_j - \widetilde H_j|\geq \eta\right)=  \mathbb P\left (\widehat{g}(\mathcal K) -g(\mathcal K)\geq (\mathcal K^{2\eta}-1) (g(\mathcal K)+ \mathcal K^{-2H_j})  \right)\\+  \mathbb P\left (\widehat{g}(\mathcal K) -g(\mathcal K)\leq (\mathcal K^{-2\eta}-1) (g(\mathcal K)+ \mathcal K^{-2H_j})  \right).
\end{multline}
Using the inequality $\log(x+1) \leq x$ with $x = \mathcal K^{2\eta} - 1$, the identity
 $ 1-\mathcal K^{-2\eta} = \mathcal K^{-2\eta}  \{\mathcal K^{2\eta}-1\}$, and  Lemma \ref{concentration22}, we get
\begin{multline}
\mathbb P \left(|\widehat H_j - \widetilde H_j|\geq \eta\right)\leq 2 \exp \left ( -\frac{\mathcal K^D -1}{ \varrho } \times \frac{4\eta^2 \log^2(\mathcal K)\mathcal K^{-4\eta} \{g(\mathcal K) + \mathcal K^{-2H_j}\}^2}{ 2g(\mathcal K) + 2\eta\log(\mathcal K)\mathcal K^{-2\eta} \{g(\mathcal K) + \mathcal K^{-2H_j}\} } \right) \\
+2\exp \left (- \mathfrak k_1 \sqrt{\mathcal N}\sqrt{\varrho- g(\mathcal K)} \right)+ 2 N\exp\left(- 2C_0^2\sqrt{\underline C_M}\mathcal K^{-2D}\mathfrak m^{1/2}\right).
\end{multline}
Finally, choosing 
$\varrho =\log^2(N) \{\mathfrak k_1 ^2 \mathcal N\}^{-1} + g(\mathcal K)$, noting that $ g(\mathcal K, \varrho)\leq g(\mathcal K)$, and using the fact that the map $x\mapsto x^2/(a+x)$ with $a>0$ is strictly increasing, 
we get  
\begin{multline}
\mathbb P \left(|\widehat H_j - \widetilde H_j|\geq \eta\right)\leq 2\exp \left ( -  \frac{\mathfrak k_1^2 (\mathcal K^{D}-1) \mathcal N}{ \log^2(N)+\mathfrak k_1^2\mathcal N g(\mathcal K)}\times \frac{2\eta^2 g(\mathcal K)\log^2(\mathcal K) \mathcal K^{-4\eta}}{ 1 + \eta \log(\mathcal K)\mathcal K^{-2\eta}  } \right)\\
+\frac{2}{N}+ 2 N\exp\left(- 2C_0^2\sqrt{\underline C_M}\mathcal K^{-2D}\mathfrak m^{1/2}\right).
\end{multline}
\end{proof}

\begin{lemma}\label{concentration22}
Assume that conditions  (D\ref{assuD1}) to (D\ref{assuD4}),   (D\ref{zerport_main}), (H\ref{assuH2}) and (H\ref{assuH12}) are satisfied. Then, a positive constant $\mathfrak k_1$ exists such that for  $\mathcal K$ sufficiently large and any $\varrho>0$ satisfying 
$$
\varrho > g(\mathcal K, \varrho):=(\mathcal K ^D- 1)^{-1} \sum_{\mathbf k \in \mathcal X(\mathcal K)} \mathbb E\left[\widehat{b}_{\mathbf k} \mathbbm{1}_{\{ \widehat{b}_{\mathbf k} \leq \varrho\} }\right],
$$ 
it holds
\begin{multline}\label{chihaja}
\mathbb P\!\left(\!\left|\frac{1}{\mathcal K ^D\!-\! 1}\! \!\sum_{\mathbf k \in \mathcal X(\mathcal K)}\! \!\widehat{b}_{\mathbf k} \!- \!\mathbb E\left[\widehat{b}_{\mathbf k} \right]\! \right| \! \geq \!\eta\! \right)\!\! \leq\!2\exp \!\left (\!\!-  \frac{(\mathcal K ^D\!-\!1)\eta^2}{ \varrho\{2g(\mathcal K, \varrho) \!+ \!\eta\}}\!\! \right) 
\\ +2\exp \left (\!- \mathfrak k_1 \sqrt{\mathcal  N(\varrho-\! g(\mathcal K, \varrho))} \right)+   2 N\exp\left(- 2\sqrt{\underline C_M}C_0^2\mathcal K ^{-2D}\mathfrak m^{1/2}\right)   .
\end{multline}
\end{lemma}

\medskip

\begin{proof}[Proof of Proposition \ref{pour_auj}]
By definition, $\{ | j_0/\mathcal J \!-\! \widehat j_0/\mathcal J|\! \geq \!\mathcal J^{-1}\}= \left\{\widehat j_0 \ne j_0\right\}$. Next, we have
\begin{equation}
\left\{\widehat j_0 \ne j_0\right\}
\subset \left (\bigcup_{j=1}^{ j_0-1} \{ |H_j - \widehat H_j| >\log(N)^{-\tau^\prime}\} \right)\cup \left(\bigcup_{j=j_0+1} ^{\mathcal J}\{ |H_j - \widehat H_j| \leq \log(N)^{-\tau^\prime}\} \right)
\end{equation}
and  
\begin{multline}
\mathbb P\left(\widehat j_0 \ne j_0\right) \leq \sum_{j=1}^{ j_0-1} \mathbb P\left(|H_j - \widehat H_j| >\log(N)^{-\tau^\prime}\right)  + \sum_{j=j_0+1} ^{\mathcal J}\mathbb P \left( |H_j - \widehat H_j| \leq \log(N)^{-\tau^\prime}\right)\\
\leq \sum_{j=1}^{ j_0-1} \mathbb P\left(|\widetilde H_j - \widehat H_j| > \frac{1}{2}\log(N)^{-\tau^\prime}\right) + \sum_{j=1}^{ j_0-1} \mathbb P\left(|H_j - \widetilde H_j| >\frac{1}{2} \log(N)^{-\tau^\prime}\right) \\
+\sum_{j=j_0+1} ^{\mathcal J}\mathbb P \left( |H_j - \widehat H_j| \leq\log(N)^{-\tau^\prime}\right).
\end{multline}
By \eqref{bias_Hj} and since $\mathcal J= \log(N)^{r^\prime}$, we have
    \begin{equation}\label{bias_Hjb}
\max_{1\leq j \leq j_0-1}	|H_j - \widetilde{H}_j| \leq\frac{1}{2\log(\mathcal K_\tau)} \mathcal K_\tau^{2H_{j_0}}g(\mathcal{K_\tau}) \lesssim \frac{1}{\log(\mathcal K_\tau)} \mathcal J^{-r} \mathcal K_\tau^{-1/\mathcal J}\lesssim \log(N)^{-r r^\prime - \tau}.
\end{equation}
This, and because we impose $ -r r^\prime - \tau< -\tau^{\prime}$, imply that, for $N$ sufficiently large,
$$
\sum_{j=1}^{ j_0-1} \mathbb P\Big(| H_j - \widetilde H_j| > \frac{1}{2}\log(N)^{-\tau^\prime}\Big) =0.
$$
Next, Proposition \ref{pour_demain}, applied with $j\in\{1,\ldots, j_0-1\}$, implies
\begin{multline}
\sum_{j=1}^{ j_0-1} \mathbb P \Big(|\widehat H_j - \widetilde H_j|\geq  \frac{1}{2}\log(N)^{-\tau^\prime}\Big)\leq\frac{2(j_0-1)}{N}+ 2 N\exp\left(- 2C_0^2\sqrt{\underline C_M}\mathcal K^{-2D}\mathfrak m^{1/2}\right)\\
+ 2(j_0-1)\exp \left ( -  \frac{\mathfrak k_1^2 (\mathcal K_{\tau}^{D}-1) \mathcal N}{ \log^2(N)+\mathfrak k_1^2\mathcal N g(\mathcal K_{\tau})}\times \frac{ g(\mathcal K_{\tau})\log(N)^{2\tau-2\tau^\prime} \mathcal K_{\tau}^{-2\log(N)^{-\tau^\prime}}}{ 2 + \log(N)^{\tau- \tau^\prime}\mathcal K_{\tau}^{-\log(N)^{-\tau^\prime}}  } \right).
\end{multline}
Since $  \log^2(N)\ll \mathcal N g(\mathcal K_{\tau})$ and $\log(N)^{\tau- \tau^\prime}\mathcal K_{\tau}^{-\log(N)^{-\tau^\prime}}\!\!= \log(N)^{\tau- \tau^\prime}\exp( - \log(N)^{\tau - \tau^\prime}) \ll 1$, we obtain 
\begin{multline}\label{boundj45}
\sum_{j=1}^{ j_0-1} \mathbb P \Big(|\widehat H_j - \widetilde H_j|\geq  \frac{1}{2}\log(N)^{-\tau^\prime}\Big)\leq\frac{2(j_0-1)}{N}+ 2 (j_0-1)N\exp\left(- 2C_0^2\sqrt{\underline C_M}\mathcal K^{-2D}\mathfrak m^{1/2}\right)\\
+ 2(j_0-1)\exp \left ( -  \mathfrak c \log(N)^{2\tau-2\tau^\prime} \exp(D\log(N)^\tau - 2\log(N)^{\tau - \tau^\prime}) \right).
\end{multline}

for some constant $\mathfrak c$. 
Finally, we know that the quantity $g(\mathcal K_\tau)$ in \eqref{g_cal} does not decrease faster than  $\mathcal{K}_\tau^{-2(\alpha+\epsilon)}$, for any small $\epsilon >0$. Thus, for $H_j>\alpha$ and $\mathcal K_\tau$ sufficiently large, we have $\widetilde H_j - H_j\geq \mathfrak q >0$, for some constant $\mathfrak q$. Therefore, for any $j\in \{j_0+1 , \ldots, \mathcal J\},$ 
\begin{multline}
\mathbb P \left( |H_j - \widehat H_j| \leq\log(N)^{-\tau^\prime}\right)=\mathbb P \left( |H_j - \widehat H_j| \leq  \log(N)^{-\tau^\prime}, | \widehat H_j - \widetilde H_j| < \frac{1}{2} \log(N)^{-\tau^\prime}\right)\\
+\mathbb P \left( |H_j - \widehat H_j| \leq \log(N)^{-\tau^\prime}, | \widehat H_j - \widetilde H_j| \geq \frac{1}{2} \log(N)^{-\tau^\prime}\right)\\
\leq\underbrace {\mathbb P \left( H_j - \widetilde H_j \leq\frac{1}{2}  \log(N)^{-\tau^\prime}\right)}_{ = 0 \text{ for } \mathcal K_\tau\text{ sufficiently large}}+\mathbb P \left( | \widehat H_j - \widetilde H_j| \geq \frac{1}{2} \log(N)^{-\tau^\prime}\right).
\end{multline}
Then, applying  Proposition \ref{pour_demain} with $j\in\{j_0+1,\ldots, \mathcal J\}$, we deduce the same type of bound as in \eqref{boundj45}, with $j_0$ replaced by $\mathcal J-j_0$. 
Gathering facts, we obtain 
\begin{multline}
\mathbb P\left(\widehat j_0 \ne j_0\right) \leq\frac{2(\mathcal J-1)}{N}+ 2 (\mathcal J-1)N\exp\left(- 2C_0^2\sqrt{\underline C_M}\mathcal K^{-2D}\mathfrak m^{1/2}\right)\\
+ 2(\mathcal J-1)\exp \left ( -  \mathfrak c \log(N)^{2\tau-2\tau^\prime} \exp(D\log(N)^\tau - 2\log(N)^{\tau - \tau^\prime}) \right).
\end{multline}
for some constant $\mathfrak c$. Now, the proof of Proposition \ref{pour_auj} is complete.
\end{proof}

\bibliographystyle{apalike}
\bibliography{clean_refs.bib}

\end{document}